\documentclass[12pt]{amsart}
\usepackage{amssymb,epsfig,amsmath,latexsym,amsthm}
\usepackage{graphicx}
\usepackage{enumitem}
\theoremstyle{plain}
\newtheorem{theorem}{Theorem}[section]
\newtheorem{lemma}[theorem]{Lemma}

\newtheorem{proposition}[theorem]{Proposition}
\newtheorem{corollary}[theorem]{Corollary}
\newtheorem{question}[theorem]{Question}
\newtheorem{conjecture}[theorem]{Conjecture}
\newtheorem{conjecture/question}[theorem]{Conjecture/Question}
\newtheorem{problem}[theorem]{Problem}
\theoremstyle{definition}
\newtheorem{definition}[theorem]{Definition}
\newtheorem{remark}[theorem]{Remark}

\newtheorem{remarks}[theorem]{Remarks}

\DeclareMathOperator{\length}{length}
\DeclareMathOperator{\area}{area}

\DeclareMathOperator{\volume}{volume}
\DeclareMathOperator{\genus}{genus}

\DeclareMathOperator{\inte}{int}

\DeclareMathOperator{\injrad}{injrad}

\DeclareMathOperator{\vol}{vol}

\newcommand{\BH}{\mathbb H}
\newcommand{\BR}{\mathbb R}

\newcommand{\BN}{\mathbb N}

\newcommand{\BZ}{\mathbb Z}

\newcommand{\mF}{\mathcal{F}}
\newcommand{\mG}{\mathcal{G}}
\newcommand{\mH}{\mathcal{H}}

\newcommand{\mT}{\mathcal{T}}

\usepackage{epsfig,color}

\makeatother

\theoremstyle{definition}

\def\fnum{equation} 
\newtheorem{Thm}[\fnum]{Theorem}

\numberwithin{equation}{section}

\newcommand{\nn}{{\bf{n}}}
\newcommand{\Ric}{{\text{Ric}}}

\def\RR{{\mathbf R}}

\begin{document}

\title{Geometric methods in Heegaard theory}

\author{Tobias Holck Colding}\address{Department of Mathematics\\ MIT\\Cambridge, MA 02139-4307}

\author{David Gabai}\address{Department of Mathematics\\Princeton
University\\Princeton, NJ 08544}

\author{Daniel Ketover}


\thanks{Version 0.59 February 2, 2020\\The first author was partially supported by NSF Grants DMS 1404540,  DMS 181214, FRG DMS 0854774, the second by DMS-1006553, DMS-1607374 and NSF FRG grant DMS-0854969 and the third by an NSF Postdoctoral Fellowship.}
 
 \email{colding@math.mit.edu, gabai@math.princeton.edu and dketover@math.princeton.edu}
\begin{abstract} We survey some recent geometric methods for studying Heegaard splittings of 3-manifolds \end{abstract}

\maketitle

\setcounter{section}{-1}

\section{Introduction}\label{S0} A \emph{Heegaard splitting} of a closed orientable 3-manifold $M$ is a decomposition of $M$ into two handlebodies $H_0$ and $H_1$  which intersect exactly along their boundaries.   This way of thinking about 3-manifolds (though in a slightly different form) was discovered by Poul Heegaard in his  forward looking 1898 Ph. D. thesis \cite{Hee1}.  He wanted to classify 3-manifolds via their diagrams.  In his words\footnote{See the first paragraph of chapter 11 of \cite{Hee1}}:   \emph{Vi vende tilbage til Diagrammet.  Den Opgave,  der burde l\o ses, var at reducere det til en Normalform; det er ikke lykkedes mig at finde en saadan, men jeg skal dog frems\ae tte nogle Bem\ae rkninger angaaende Opgavens L\o sning.} ``We will return to the diagram.  The problem that ought to be solved was to reduce it to the normal form; I have not succeeded in finding such a way but I shall express some remarks about the problem's solution."  For more details and a historical overview see \cite{Go}, and \cite{Zie}.  See also the French translation \cite{Hee2}, the English translation \cite{Mun} and the partial English translation \cite{Prz}.  See also the encyclopedia article of Dehn and Heegaard \cite{DH}.

In his 1932 ICM address in Zurich \cite{Ax}, J. W. Alexander  asked  \emph{to determine in how many essentially different ways a canonical region can be traced in a manifold} or in modern language \emph{how many different isotopy classes are there for splittings of a given genus}.  He viewed this as a step towards Heegaard's program. 

Most of this survey paper is about the geometric and topological techniques that have been recently used towards answering Alexander's problem for irreducible splittings of non Haken hyperbolic manifolds.    In particular we give the key ideas behind the proofs of the following results.

\begin{theorem}\cite{CG}  If $M$ is a closed non-Haken 3-manifold, then there exists an effective algorithm to produce a finite list of Heegaard surfaces which contains all the irreducible Heegaard splittings up to isotopy.\end{theorem}

This result is an effective version of theorems of Tao Li  \cite{Li2}, \cite{Li3}.  To essentially answer Alexander's question for irreducible splittings one must root out the reducible splittings and duplicate splittings from this list.  Achieving that gives the following. 

\begin{theorem} \cite{CGK} Let $N$ be a closed non-Haken hyperbolic 3-manifold. There exists an effectively constructible set $S_0, S_1, \cdots, S_n$ such that if $S$ is an irreducible Heegaard splitting, then $S$ is isotopic to exactly one $S_i$.\end{theorem}

Both results make crucial use of the recent resolution of the circa 1987 Pitts-Rubinstein conjecture.

\begin{conjecture}[Pitts-Rubinstein]
A strongly irreducible Heegaard surface in a Riemannian three-manifold is either 
\begin{enumerate}
\item isotopic to a minimal surface of index at most $1$ 
\item isotopic after a single compression to the boundary of a tubular neighborhood of a stable one-sided Heegaard surface.  
\end{enumerate}
\end{conjecture}

\begin{theorem}[K-Liokumovich-Song \cite{KLS}]
The Pitts-Rubinstein conjecture is true. 
\end{theorem} 
\vskip 8pt

For excellent earlier surveys on Heegaard theory see  \cite{Zie},\cite{Sc2}, \cite{So}, \cite{SSS}.  See also the bibliography which includes many references to papers not discussed in this survey.
\vskip 10 pt
\noindent\emph{Acknowledgement}  We thank the referee for his/her constructive comments on this paper.

\section{Summary of this paper} 

Section \S2 gives basic definitions and now classical results about Heegaard splittings.  \S3 surveys min-max methods to construct minimal surfaces from $k$-parameter sweepouts.  Section \S4 outlines the recent resolution of the Pitts-Rubinstein conjecture, asserting roughly speaking that strongly irreducible Heegaard surfaces can be isotoped to index at most $1$ minimal surfaces.  Section \S5 outlines the proof that there are finitely many $\eta$-negatively curved branched surfaces that carry all index-$\le 1$ minimal surfaces in hyperbolic 3-manifolds and hence the strongly irreducible Heegaard surfaces in these manifolds.  \S6 is about mean convex foliations.  \S7 and \S8 outline a proof of an effective version of Tao Li's finiteness theorem for irreducible splittings of non-Haken 3-manifolds.  \S9 outlines an effective algorithm for enumerating without duplication the irreducible Heegaard splittings of a hyperbolic non-Haken 3-manifold. In \S10 we state some open questions.  In the Appendix we very briefly mention some speculative thoughts of the first two authors on the geometry of hyperbolic handlebodies $V$ such that  $\partial V$ is an index-1 minimal surface,  $\vol(V)>>0$ and $V$ has a uniformly bounded Cheeger constant.

There are many results and techniques that we do not discuss in this paper, though they appear in the references.  E.g. geometric topological techniques for putting Heegaard splittings into a normal form with respect to some other structure, normal and almost normal surfaces, connections with the curve complex, hierarchies, the Rubinstein - Scharlemann graphic, thin position, Tao Li's work on geometric vs algebraic rank of closed 3-manifolds, and some aspects of minimal surfaces.  In addition there are unmentioned algebraic techniques for distinguishing Heegaard splittings, e.g. invariants of a certain double coset of the mapping class group of the Heegaard surface and Nielsen Equivalence.

\section{Basic Definitions and Facts}  

\begin{definition}  An \emph{effective} algorithm is one that produces an output as a computable function of the initial data.  \end{definition}

It is interesting to note that the algorithms behind the results in this paper are elementary and combinatorial, yet the proofs that they work often require sophisticated geometric arguments, e.g. Theorem \ref{main} requires a 2-parameter sweepout argument and a mulit-parameter min-max argument.
\vskip 10pt

For the definitions of basic notions related to branched surfaces see \cite{O} and \S 1 \cite{CG}.
\begin{definition}  A \emph{Heegaard splitting} of a closed orientable 3-manifold $M$ consists of an ordered pair  $(H_0, H_1)$ of handlebodies whose union is $M$ and whose intersection is their boundaries.  This common boundary $S$ is called a \emph{Heegaard surface}.  Two Heegaard splittings $(H_0, H_1)$, $(H'_0, H'_1)$ are \emph{isotopic} if there exists an ambient isotopy of $M$ taking $H_0$ to $H'_0$.  The \emph{genus} of the splitting is the genus of the splitting surface and the \emph{Heegaard genus} of $M$ is the smallest $g$ for which there is a Heegaard splitting of that genus.  \emph{Stabilization} is the process of increasing a genus-n splitting to a genus-$n+1$ splitting by adding a handle in the obvious manner.  The inverse operation is called \emph{destabilization}.  A splitting is \emph{irreducible} if it is not a stabilization.\end{definition}

\begin{remarks}  i)  In this paper we consider equivalence classes up to isotopy.  One can also consider splittings up to orientation preserving homeomorphism, or up to isotopy or homeomorphism, orientation preserving or not, without regard to whether the sides are preserved.  

ii)  Here are a few results.  The 3-dimensional Schoenflies theorem of Alexander implies that the 3-sphere has a unique splitting of genus-0.  Waldhausen \cite{Wa1} showed that every positive genus Heegaard splitting of the 3-sphere is a stabilization, hence the genus-0 splitting is the unique one up to stabilization.  Note that there is an isotopy that switches the sides of the splitting.  Engmann \cite{Eng} and Birman \cite{Bir} independently first showed that there exist manifolds with non homeomorphic Heegaard splittings, even for allowing the switching of sides.  Bonahon - Otal \cite{BoOt} and Hodgeson - Rubinstein \cite{HR} independently classified Heegaard splittings of lens spaces.  In particular, the genus-1 Heegaard surface is unique up to isotopy, but some manifolds have isotopies that switch sides and some do not.  Further, all higher genus splittings are stabilizations.   \end{remarks}

\noindent\textbf{Existence of Heegaard Splittings} Heegaard \cite{Hee1} proved that every triangulated 3-manifold has a Heegaard splitting, more or less, by showing that after removing a neighborhood of a point it deformation retracts to a 2-complex.  The boundary of a neighborhood of the 1-skeleton of this 2-complex gives a Heegaard surface.  Alternatively, the boundary of the  neighborhood of the 1-skeleton of the triangulation is a Heegaard surface.  More generally Moise \cite{Mo} proved that every 3-manifold has a triangulation and hence a Heegaard splitting.

\vskip 8pt

\noindent\textbf{Uniqueness up to Stabilization}  The Reidemeister-Singer theorem, see \cite{Re}, \cite{Si}, \cite{Cr}, \cite{Sie},\cite{Lau} asserts that given two Heegaard splittings of a closed 3-manifold, then after stabilizing each a finite number of times, the resulting Heegaard splittings are isotopic.  A consequence of this is that associated to a 3-manifold is the \emph{tree of Heegaard splittings}. Here vertices are the isotopy classes of splittings and there is a directed edge from $v$ to $v'$ if $v $ is obtained from $v'$ by a single stabilization.

\vskip 8pt 

\noindent\textbf{Connect Sums}  Haken \cite{Ha3} showed that if $S$ is a Heegaard surface in $M=M_1\#M_2$, then a summing sphere $Q$ can be isotoped to intersect the Heegaard surface $S$ in a single circle and hence after isotopy the Heegaard splitting  \emph{restricts} to a Heegaard splitting of each with Heegaard surfaces $S_1$ and $S_2$ and hence genus is additive under connect sum.  Bachman \cite{Ba2} and Qiu - Scharlemann \cite{QS} further showed that if $S$ is reducible, then one of $S_1$ and $S_2$ is reducible.
\begin{definition} (Casson - Gordon \cite{CaGo1})The Heegaard splitting $\mH= (H_0, H_1)$ is \emph{weakly reducible} if there  exist essential compressing discs $D_i$ for $H_i, \ i=0,1$  such that $\partial D_0\cap \partial D_1=\emptyset$.  The splitting $\mH$ is \emph{strongly irreducible} if it is not weakly reducible.  \end{definition}

The following theorem of Casson - Gordon \cite{CaGo1} plays a central role in this paper.

\begin{theorem}  If $M$ is a closed irreducible 3-manifold with an irreducible, weakly reducible Heegaard splitting, then $M$ has an embedded incompressible surface and hence is Haken.\end{theorem}

\noindent\textbf{Compact Manifolds with Boundary}  The theory of Heegaard splittings naturally extends to connected compact manifolds with boundary.  Let $M$ be such a manifold with $\partial M$ the disjoint union of $\partial_0 M$ and $\partial_1 M$ with possible $\partial_1 M=\emptyset$.  Then a Heegaard splitting $(H_0, H_1)$ of $(M,\partial_0 M, \partial _1 M)$ is a decomposition where each $H_i$ is a compression body with $\partial_-(H_i)=\partial_iM$ and $\partial_+(H_0)=\partial_+(H_1)$.    The above results and their proofs directly extend to Heegaard splittings of compact manifolds.  This paper will restrict itself to closed manifolds, though most results extend to the compact case.  See the earlier survey papers and many of the papers in the references for results on manifolds with boundary.

\section{Minimal surfaces and min-max theory}
A closed surface $\Sigma$ in a $3$-manifold $N$ is a \emph{minimal surface} if the first variation of area is zero for all variations. Equivalently, $\Sigma$ is minimal if its mean curvature is identically equal to zero.  On a small enough scale, a minimal surface minimizes area.  There may however be global deformations which bring the area down.  To that end, on a minimal surface we can consider the second variation of area.  Let $\phi$ be a function on $\Sigma$ and $\nn$ the unit normal of $\Sigma$.  When $\Sigma$ is minimal and $\Sigma_s$ is a normal variation of $\Sigma$ with $\Sigma_0=\Sigma$ and variational vector field $\phi\,\nn$, then the second variation of area is
\begin{align}
\frac{d^2}{ds^2}\bigg\lvert_{s=0}\,\mbox{area} (\Sigma_s)=-\int_{\Sigma}\phi\,L\,\phi\, .
\end{align}
Here $L$ is the second variational operator given by that 
\begin{align}
L\,u=\Delta_{\Sigma}\,u+|A|^2\,u+\Ric_N(\nn,\nn)\,u\, ,
\end{align}
$A$ is the second fundamental form of $\Sigma$, and $\Ric_N(\nn,\nn)$ is the Ricci curvature of $N$ in the normal direction.  The operator $L$ is a Schr\"odinger operator and has only finitely many negative eigenvalues since $\Sigma$ is closed.  The \emph{Morse index} of a minimal surface is the number of negative eigenvalues of $L$.    It is said to be \emph{stable} when the index is zero and \emph{strictly stable} if the index is zero and $0$ is not an eigenvalue.  

Meeks-Simon-Yau \cite{MSY} proved that one can always minimize area in some non-trivial isotopy class to obtain a stable minimal surface.  Many three-manifolds (such as the three-sphere) contain no such non-trivial classes, and min-max methods are necessary to construct higher Morse index critical points of the area functional.  

The idea of min-max theory is the following. For some fixed $g$, consider the space of embedded genus $g$ surfaces in a Riemannian three-manifold.  On this space one can define the area functional, associating to each surface its area.  As in analogy with finite-dimensional Morse theory, non-trivial topology in the space of all genus $g$ surfaces should force the existence of critical points of the area functional, or minimal surfaces.  The simplest such genus $g$ families arise from considering Heegaard foliations. This infinite dimensional Morse theory was begun by Almgren in the 60s, completed by Pitts \cite{Pi} and Simon-Smith \cite{SS} in the 80s.  There were further works by Pitts-Rubinstein (\cite{PR1} \cite{PR2}) and Frohman-Hass \cite{FH}. Let us now give more details.  

Let $X^k$ be a $k$-dimensional manifold with boundary.   A \emph{sweepout parameterized by $X$} is a family of closed sets $\left\{\Sigma_t\right\}_{t\in X}$ continuous in the Hausdorff topology such that 

\begin{enumerate}
\setlength{\itemsep}{1pt}
  \setlength{\parskip}{0pt}
  \setlength{\parsep}{0pt}

\item $\Sigma_t$ is an embedded smooth surface for $t$ in the interior of $X$
\item $\Sigma_t$ varies smoothly for $t$ in the interior of $X$
\item For $t\in\partial X$, $\Sigma_t$ is a (possibly empty) smooth surface together with some $1$-dimensional arcs.  
\end{enumerate}

The simplest example of such a sweepout is a Heegaard foliation, where $X=I$ is the unit interval, and the endpoints of $I$ correspond to spines in the two handlebodies.  

If $\Lambda$ is a collection of sweepouts, we say that the set $\Lambda$ is \emph{saturated} if given a map $\phi\in C^{\infty}(I\times M,M)$ such that $\phi(t,-)\in\text{Diff}_{0} M$ for all $t\in I$, and a family  $\left\{\Sigma_t\right\}_{t\in X}\in\Lambda$, we have  $\left\{\phi(t,\Sigma_t)	\right\}_{t\in I}\in\Lambda$.  Denote by $\Lambda$ the smallest saturated family of sweepouts containing $\{\Sigma_t\}_{t\in X}$.

The width associated to $\Lambda$ is defined to be
\begin{equation}\label{w}
W:=W(M,\Lambda_H)=\inf_{\left\{\Sigma_t\right\}\in\Lambda}\sup_{t\in I} \mathcal{H}^2(\Sigma_t),
\end{equation}
\noindent
where $\mathcal{H}^2$ denotes $2$-dimensional Hausdorff measure. 

Suppose that \begin{equation} W > \sup_{t\in\partial X} \mbox{area}(\Sigma_t).\end{equation}  In this case we say that the sweep-out $\{\Sigma_t\}_{t\in X}$ is \emph{non-trivial}.  

We can now define a sequence of sweepouts in the saturation ``pulled tight" in the sense that their maximal areas are approaching the width. Namely, a \emph{minimizing sequence} is a sequence of families $\left\{\Sigma^n_t\right\}\in\Lambda_H$ such that
\begin{equation} 
\lim_{n\rightarrow\infty}\sup_{t\in[0,1]} \mbox{area}(\Sigma^n_t)=W.
\end{equation}
\
A \emph{min-max sequence} is then a sequence of slices $\Sigma^n_{t_n}$, $t_n\in (0,1)$ such that
\begin{equation} 
\mbox{area}(\Sigma^n_{t_n})\rightarrow W.
\end{equation}

The main result of Simon-Smith 80s \cite{SS} (c.f. \cite{CD}, \cite{DP}) is that some min-max sequence converges to a minimal surface (potentially disconnected, and with multiplicities):
\begin{theorem}[Simon-Smith]\label{SS}
Let $M$ be a closed oriented Riemannian $3$-manifold and $\Sigma_t$ $t\in X$ a $k$-parameter sweepout, and $\Lambda$ the smallest saturated set containing $\{\Sigma_t\}_{t\in X}$.  
Suppose 
\begin{equation}\label{isnontrivial}
W > \sup_{t\in\partial X} \mbox{area}(\Sigma_t).
\end{equation}

Then some min-max sequence $\Sigma_{t_i}^i$ converges as varifolds to $\sum_{j=1}^k n_j \Gamma_j$, where $\Gamma_j$ are smooth embedded pairwise disjoint minimal surfaces and where $n_j$ are positive integers.  Moreover, \begin{equation}
W=\sum_{j=1}^k n_j \mbox{area}(\Gamma_j).
\end{equation}
\end{theorem}

One also has the following genus bounds for the limiting minimal surfaces, expressing that the limit is achieved after surgeries.  Namely, after finitely many compressions on a min-max sequence, one obtains a surface isotopic to $n_i$ parallel copies about each $\Gamma_i$:  
\begin{theorem}[Genus bounds \cite{Ke}]\label{genusbounds}
The genus of the limiting minimal surface (in the notation of the previous theorem) can be controlled as follows:
\begin{equation}\label{gb}
\sum_{i\in O} n_ig(\Gamma_i) + \frac{1}{2}\sum_{i\in N} n_i(g(\Gamma_i)-1)\leq g,
\end{equation}
where $O$ denotes the set of $i$ such that $\Gamma_i$ is orientable, and $N$ the set of $i$ such that $\Gamma_i$ is non-orientable, and $g(\Gamma)$ denotes the genus of $\Gamma$.  The genus of a non-orientable surface is the number of cross-caps one must attach to a two-sphere to obtain a homeomorphic surface.
\end{theorem}

Recently Marques-Neves \cite{MN} obtained the following upper Morse index bounds:

\begin{proposition}[Upper Index Bounds \cite{MN}]\label{indexbounds}
Suppose all components of the min-max limit are orientable.  Then 
\begin{equation}
\sum_{i\in O} index(\Gamma_i) \leq k, 
\end{equation}
where $k$ denotes the dimension of the parameter space $X$.  
\end{proposition}

\section{Pitts-Rubinstein Conjecture}

Given a Heegaard splitting of a three-manifold, a natural question is when one can isotope the Heegaard surface to be a minimal surface.   A Heegaard surface gives a sweepout of a three-manifold (parameterized by $X$ equal to the unit interval $I$ where at the two points of $\partial I$ the surface is fixed to degenerate to one-dimensional spines of the handlebodies).  It is a consequence of the isoperimetric inequality that $W>0$.   Thus \eqref{isnontrivial} is satisfied and this family is \emph{non-trivial} and the Min-Max Theorem \ref{SS} applies.    

The difficulty is that the minimal surface obtained may contain several connected components, some with positive integer multiplicities.   In fact, if one starts with a random stabilized Heegaard surface, one does not expect to produce a minimal surface isotopic to the surface.  For instance, in round $S^3$, if one runs a min-max procedure with respect to the stabilized genus $1$ splitting by Proposition \ref{indexbounds} one obtains a minimal surface of index at most $1$, which must be an equator of genus $0$.

However, in the 80s Pitts-Rubinstein conjectured that if one begin with a \emph{strongly irreducible} splitting, one has much better control.  Namely, 

\begin{conjecture}[Pitts-Rubinstein]\label{pitts}
A strongly irreducible Heegaard surface in a Riemannian three-manifold is either 
\begin{enumerate}
\item isotopic to a minimal surface of index at most $1$ 
\item isotopic after a single compression to the boundary of a tubular neighborhood of a stable one-sided Heegaard surface.  
\end{enumerate}
\end{conjecture}



Recall that a \emph{one-sided Heegaard surface} is a surface $\Sigma$ in a three-manifold whose complement is an open handlebody.  The simplest example is $\mathbb{RP}^2\subset \mathbb{RP}^3$ as $\mathbb{RP}^3\setminus \mathbb{RP}^2$ is a three-ball.  After a single compression on a Heegaard torus in $\mathbb{RP}^3$ one obtains a two-sphere bounding an $I$-bundle about an $\mathbb{RP}^2$.  

The third-named author together with Y. Liokumovich and A. Song recently proved Conjecture \ref{pitts}:

\begin{theorem}[K-Liokumovich-Song \cite{KLS}]\label{pittsrubinstein}
The Pitts-Rubinstein conjecture is true. 
\end{theorem} 

It is shown in \cite{KLS} that the dichotomy in Theorem \ref{pittsrubinstein} is \emph{sharp} in that one can find a metric on $\mathbb{RP}^3$ for which (2) and not (1) occurs.  Namely, there exist metrics on $\mathbb{RP}^3$ of positive scalar curvature containing \emph{no} index $1$ minimal Heegaard tori.    
  
The assumption of \emph{strong irreducibility} in Theorem \ref{pittsrubinstein} is essential.  The three-torus has a genus $3$ irreducible but not strongly irreducible splitting.  In flat tori given by quotients of $\mathbb{R}^3$ by nearly degenerate lattices, however, Ritore-Ros \cite{RR} have shown that there \emph{do not exist} index $1$ minimal surfaces of genus $3$ (and nor can case (2) occur).    
 
Why is strong irreducibility so important?   Firstly, essential compressions on such a surface can only be into \emph{one} of the handlebodies. Secondly, if one runs the min-max process relative to such splittings, all orientable surface obtained (except for two-spheres) must have multiplicity $1$:

 \begin{proposition}[Multiplicity One]\label{mult1}
Suppose a min-max procedure is performed relative to a strongly irreducible splitting to obtain minimal surfaces $\Gamma_1, ..., \Gamma_k$ occurring with multiplicities $n_1, ...n_k$. If any $\Gamma_i$ is orientable and of positive genus, then $n_i=1$.  
\end{proposition}
\begin{proof}
By Theorem \ref{genusbounds} it follows that after finitely many compressions on $\Sigma$ we obtain for each $i$, $n_i$ parallel sheets $S_1$, ... $S_{n_i}$ in a neighborhood of $\Gamma_i$.  By strong irreducibility, all essential compressions on $\Sigma$ must be into the same handlebody $H_1$ giving rise to handlebodies of smaller genera $J_1$, ... $J_k$.   We claim $n_i=1$.   If $n_i\geq 3$, then either the region bounded between $S_1$ and $S_2$ is a handlebody among the $J_i$ or else the region between $S_2$ and $S_3$ is a handlebody among the $J_i$.   But both of these regions are homeomorphic to $\{\text{surface}\}\times I$, and thus not handlebodies.  Thus $n_i\leq 2$.  If $n_i=2$, then the three-manifold consists of three components: $J_1$, $J_2$  and the region between $S_1$ and $S_2$ (where $\partial J_1=S_1$ and $\partial J_2=S_2$).  It follows from Scharleman-Thompson's classification of Heegaard splittings of $\{\text{surface}\}\times I$ \cite{ST1} that the only way to obtain an irreducible splitting by adding handles back to $S_1\cup S_2$ is by attaching a single vertical handle joining $S_1$ to $S_2$.  But if this is the case one can find two disjoint curves on the resulting surface giving a weak reduction.   Thus $n_i=1$. 
\end{proof}

\begin{remark}
Strong irreducibility is essential here.  In the genus three splitting of $T^3$ one can perform two compressions, one into each handlebody and obtain two parallel tori.  Thus the multiplicity of a torus obtained in a min-max process might be two.  
\end{remark}

Let us now explain Pitts-Rubinstein's sketch of Conjecture \ref{pitts}.   Since in Theorem \ref{SS} the minimal surface obtained by min-max methods could consist of several components, the idea of Pitts-Rubinstein was to iterate the min-max procedure until one obtained a surface isotopic to $H$. Roughly speaking, their argument was as follows (cf. \cite[Theorem 1.8]{Ru})).  By strong irreducibility, any degeneration of the min-max sequence could only be along compressions into \emph{one} of the handlebodies and the positive genus surfaces have multiplicity $1$ by Proposition \ref{mult1}.   If such degeneration occurs, one then can remove the handlebodies bounded by the several minimal surfaces to obtain a manifold with minimal boundary $M'$.   As one of the minimal boundary components should have index $1$, one could minimize area for the unstable component of $\partial M'$ into $M'$ to obtain a new manifold $M''$ with stable boundary.  One then applies min-max to the compression body $M''$ and iterates.  Since $M''$ has stable boundary, and the min-max limit should always have an unstable component, at each stage of the iteration the manifold shrinks.   If the process does not stop, one obtains infinitely many nested minimal surfaces with bounded genus, giving rise to a Jacobi field.  If the metric is bumpy (which White proved is a generic condition) then this gives a contradiction.  Thus the process stops after finitely many steps at a minimal surface isotopic to the Heegaard surface.

The argument sketched by Pitts-Rubinstein was incomplete on two points.  First, they assumed that the limit is achieved after compressions.  This was proved by the third-named author in Theorem \ref{genusbounds}.  Secondly, in order to run the iteration, one needs to apply the min-max theorem to a subdomain $M''$ of the manifold with stable minimal boundary $\partial M''$.  The key claim is that one can obtain a minimal surface in the \emph{interior} of such a subdomain.  What could go wrong is that the min-max procedure just gives rise to the boundary $\partial M''$ where some two-sphere component may have positive integer multiplicity\footnote{In hyperbolic manifolds there are no minimal two-spheres (see Lemma \ref{hyperbolicareabound}) and thus the proof of the conjecture is much simpler.}.

\begin{proposition}[Min-max with stable boundary \cite{KLS}]\label{stableboundary}
Let $M$ be a manifold with boundary consisting of two compression bodies $C_1$ and $C_2$ glued together along a strongly irreducible Heegaard surface $\Sigma$.  Suppose each component of $\partial M$ is \emph{strictly stable}.  Then $M$ contains in its \emph{interior} a minimal surface obtained from surgeries on $\Sigma$ of index at most $1$.  
\end{proposition}
In fact, to sketch the proof of Proposition \ref{stableboundary}, let us assume the manifold $M$ is as simple as possible, namely a three-ball bounded by a strictly stable minimal two-sphere $\Gamma$.
\vskip 7pt
\emph{Sketch of Proof:} Let $\{\Sigma_t\}_{t\in [0,1]}$ be a sweepout of $M$ by two-spheres so that $\Sigma_0=\Gamma$ and $\Sigma_1$ is the trivial point surface.  We consider the saturation of $\{\Sigma_t\}_{t\in [0,1]}$ and the corresponding min-max problem.   Because the boundary $\Gamma$ is strictly stable, it is not hard to show that $W>\mbox{area}(\Sigma_0)$.  Thus the sweepout is non-trivial and the Min-Max Theorem \ref{SS} applies.  We must rule out obtaining from Theorem \ref{SS} the surface $\partial M=\Gamma$ counted with multiplicity $k>1$.   We argue by contradiction.  If this happens, we will construct a competitor sweepout of $M$, also beginning at $\Gamma$ and ending at a point surface, but with all areas strictly less than $W$, contradicting the definition of $W$.  

Consider a sequence of sweepouts $\{\Sigma^i_t\}_{t\in [0,1]}$ so that $\sup_{t\in[0,1]} \mbox{area}(\Sigma^i_t)$ approaches the min-max value $W$ as $i\rightarrow\infty$.   Suppose that for $i$ large and some $t_0\in (0,1)$, the surfaces $\{\Sigma^i_t\}_{t\in [t_0-\varepsilon_i,t_0+\varepsilon_i]}$ are within a fixed $\eta>0$ neighborhood of the varifold $N=k\Gamma$.  Suppose for simplicity that $[t_0-\varepsilon_i,t_0+\varepsilon_i]$ is the only such interval.  Using the strict stability of the boundary, one can arrange \begin{equation}\label{arealow1}
\mbox{area}(\Sigma^i_{t_0-\varepsilon})< W 
\end{equation}
and
\begin{equation}\label{arealow2}
\mbox{area}(\Sigma^i_{t_0+\varepsilon})< W.
\end{equation}

The key observation is that either $\{\Sigma^i_t\}_{t\in [0,t_0-\varepsilon_i]}$ or $\{\Sigma^i_t\}_{t\in [t_0+\varepsilon_i,1]}$ must \emph{itself} be a sweepout of the entire manifold $M$ apart from a tiny tubular neighborhood of $\partial M$.   If $k$ is even, then the first is a sweep-out, and if $k$ is odd, then the second is.    Roughly speaking, the only way a family of spheres in the ball can begin at $\Gamma$ and end up close to $2\Gamma$, is if the family sweeps out the entire ball in the process.   Let's assume $k$ is even. 


We then need to construct a sweepout supported near $\partial M$ that begins at the surface $\Sigma^i_{t_0-\varepsilon_i}$ (which looks like $k$ copies of $\Gamma$) and ends at the zero or trivial point surface.  More precisely, given any $\delta>0$ we need an \emph{interpolating family} of surfaces $\{\Gamma_t\}_{t\in[0,1]}$ satisfying:

\begin{enumerate}[label=(\alph*)]
\item $\Gamma_0 = \Sigma^i_{t_0-\varepsilon_i}$ \label{1}
\item $\Gamma_1 = \mbox{ trivial point surface}$
\item$\mbox{area}(\Gamma_t)\leq \mbox{area}(\Sigma^i_{t_0-\varepsilon_i}) +\delta$. \label{3}
\end{enumerate}

Choosing $\delta$ appropriately small, we can concatenate $\{\Sigma^i_t\}_{t\in [0,t_0-\varepsilon_i]}$ and $\{\Gamma_t\}_{t\in[0,1]}$ to obtain a new sweepout of $M$ with all areas less than $W$ (thanks to \eqref{arealow1} and \eqref{arealow2}).  This gives a contradiction to the definition of width.  The conclusion is that $M$ contains in its interior a minimal surface (with index at most $1$ by earlier work of Marques-Neves \cite{MN}).  

To make the desired interpolation, a key point is that by the strict stability of $\Gamma$, by pushing off by the lowest eigenfunction of the stability operator, one can find a neighborhood of it $N_{h}(\Gamma)$ (diffeomorphic to $\mathbb{S}^2\times [0,1]$) in which $\Gamma$ is the \emph{unique} stable surface.  It would be natural to use the mean curvature flow to flow any essential two-sphere in this neighborhood toward the core stable sphere (and flow inessential spheres toward a point) but it is not yet known how to flow past singularities that the flow may develop.   

Nevertheless, we have the following:

\begin{proposition}[Mean Curvature Flow ``By Hand",  \cite{KLS}] \label{main_deformation}
Let $\Gamma$ be a strictly stable two-sphere.  Let $\Sigma$ be a smooth embedded two-sphere contained in $N_{h}(\Gamma)=\mathbb{S}^2\times [0,1]$.  For every $\delta>0$  there exists an isotopy $\Sigma_t \subset  N_{h}(\Gamma)$ with 
\begin{enumerate}
\item $\Sigma_0 = \Sigma$
\item $\Sigma_1$ is either equal to $\Gamma$ or else $\Sigma_1$ is contained in a ball 
of arbitrarily small radius
\item $\mbox{area}(\Sigma_t) \leq \mbox{area}(\Sigma)+\delta$ for all $t$. \label{eq: delta}
\end{enumerate}
\end{proposition}

The idea in the proof of Proposition \ref{main_deformation} is motivated by the following.  Suppose one asks an analogous question in $\mathbb{R}^3$.  That is, suppose one is given two embeddings of two-spheres in $\Sigma_0$ and $\Sigma_1$ in $\mathbb{R}^3$ and one seeks an isotopy joining them that does that not increase area along the way.  By Alexander's theorem, there exists \emph{some} isotopy joining them (the trouble is that areas may have to go quite high in order to achieve this).  But we can enclose both surfaces in a large ball, shrink the ball to be tiny, do the isotopy in the tiny ball, and then rescale to unit size.  

In Proposition \ref{main_deformation} we must work in $N_h(\Gamma)$ (diffeomorphic to $\mathbb{S}^2\times [0,1]$) and thus we do not have such radial isotopies to exploit.  We can however press any sphere in $N_h(\Gamma)$ \emph{arbitrarily close} to $\Gamma$ in an area-decreasing fashion (again using the lowest eigenfunction of the stability operator).  Once it is pressed close enough, we cover $\Gamma$ with balls, and can then use the squeezing trick above in each ball to simplify the surface, opening any necks and folds in the process until the surface consists of some number $m$ parallel copies of $\Gamma$ joined by a thin set of possibly badly linked and nested ``necks."  Using the Lightbulb theorem, we can then open any knotted necks to bring $m$ down and iterate until $m=1$ or $m=0$.

This completes the sketch of the interpolation result Proposition \ref{main_deformation} and thus the proof of Pitts-Rubinstein's conjecture.

\section{Minimal surfaces with index at most one in $3$-manifolds}

For an index at most one minimal surface in a $3$-manifold, near most points, the surface is stable and pointwise curvature estimates apply.  Those curvature estimates show that on a small scale the surface looks like a small almost flat piece of a plane.   However, for a surface with index $\leq 1$ there can be a single small unstable neighborhood.  If there is, then we show that in that neighborhood, the surface looks like a scaled down catenoid centered around the neck.    The two sheets of the catenoid extend and remain almost flat and graphical over each other past the small scale and up to a fixed scale.  This is the content of the next theorem \cite{CG} that is stated for a unit ball in Euclidean space but holds with obvious modifications for a sufficiently small ball in any fixed $3$-manifold.   

\begin{Thm}   \label{t:t1}
There exists $\delta$, $c > 0$ such that the following holds: Given $C > 1$, $\mu > 0$, there exists $\epsilon > 0$ so that if $\Sigma\subset B_1 \subset \RR^3$ is a compact embedded minimal surface with $\partial \Sigma\subset \partial B_1$, index one and $B_{\epsilon} \cap \Sigma$ is unstable, then there is a simple closed geodesic $\gamma\subset B_{2\,\epsilon}\cap \Sigma$ of length $\ell$ so that 
\begin{itemize}
\item $\Sigma_c\setminus An_{\delta\,\ell}(\gamma)$ consists of two graphical annuli of functions with gradient at most one.
\item $An_{C\,\ell}(\gamma)$ is $\mu$-$C^2$ close to the corresponding annulus in an rescaled catenoid with neck of length $\ell$.
\end{itemize}
\end{Thm}

In this theorem $\Sigma_c$ is the connected component of $B_c\cap \Sigma$ containing $\gamma$ and $An_{s}(S)$ are the points in $\Sigma$ with intrinsic distance at most $s$ to a subset $S$ of $\Sigma$.

   \begin{figure}[htbp]
\centering\includegraphics[totalheight=.38\textheight, width=.50\textwidth]{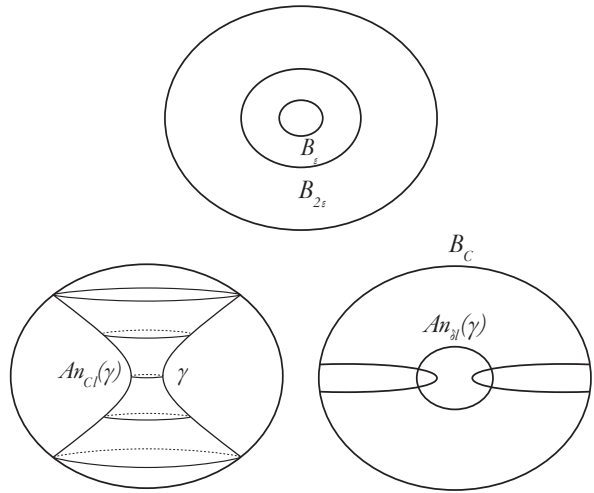}
\caption{The two scales:  (1) The structure on the scale where the index concentrates.  (2) The larger fixed scale.}   
  \end{figure}

\vskip2mm
This result gives a basic local structure for index $\leq 1$ minimal surfaces.   It is important that this theorem gives information all the way up to the fixed scale $c$ and the entire all $B_c$, where the radius $c$ is independent of $\ell$.  This theorem implies that locally such surfaces look like one or multiple almost flat and parallel sheets in addition to possibly a single pair of almost flat sheets joined by a tiny catenoidal neck.   This gives essentially immediately: Any sequence of closed embedded index one minimal surfaces has a subsequence that converges to a smooth minimal lamination with possibly one unstable leaf.  An unstable leaf can only occur if the index of the surfaces themselves does not concentrate.    When index concentrates different parts of the surfaces collapse to a sheet with multiplicity two.    This is the following theorem (cf. Corollary 2.2 \cite{CG}).

\begin{Thm}   \label{t:t2}
Let $N^3$ be a complete hyperbolic $3$-manifold and $\Sigma_i\subset N$ a sequence of closed embedded minimal surfaces with index $\leq 1$. Then a subsequence converges to a smooth minimal lamination $L$. Moreover, at most one leaf of $L$ is unstable and if $L$ is an unstable leaf, then it is isolated.
\end{Thm}

Theorem \ref{t:t2} actually holds for any $3$-manifold, compact or not, and without any assumption on the curvature.  To prove this result for complete finite volume hyperbolic $3$-manifold we needed to know that index-$1$ surfaces lie in a bounded set.  Indeed, there is the following result independently proved in \cite{CG} and \cite{CHMR}, \cite{CHMR2}.  Actually the result in \cite{CHMR} holds for any closed minimal surface.

\begin{Thm}  \label{t:t3}
Let $N^3$ be a finite volume hyperbolic $3$-manifold and $x\in N$ be a fixed point, then there exists an $R > 0$ so that any closed embedded index-$1$ minimal surface is contained in the ball $B_R(x)$.
\end{Thm}

  \begin{figure}[htbp]
\centering\includegraphics[totalheight=.38\textheight, width=.50\textwidth]{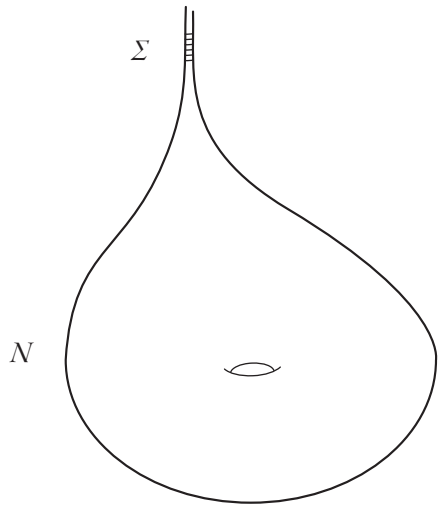}
\caption{If $\Sigma$ is closed and penetrate into a cusp, then it would begin to spiral if the index $\leq 1$.  This leads to a contradiction.}   
  \end{figure}

Here is the idea of the proof.  We will use the local structure theorem, Theorem \ref{t:t1}, to show that any closed index one minimal surface must lie within a bounded set.    Recall that the ends of a finite volume hyperbolic $3$-manifold are cusps.  Cusps are topologically a product of a torus with a half line.   The metric is such that the induced metric on the tori is flat, but as one goes further into the cusp, along the half line, the tori shrinks exponentially. 

The reason why such a surface cannot penetrate into a cusp is that if it did, since the surface is closed, there would be a point where it is deepest into the cusp.     From the convexity of horospheres it follows that the surface cannot be entirely contained in a cusp.   Following the surface around, starting at the point of deepest penetration, we move locally nearly orthogonal to the end of the cusp.  Beginning at the point of deepest penetration, as we move along, in one direction we would after a while come back further inside the thick part of the manifold; it cannot come back to itself since the surface cannot be entirely contained in the cusp.    However, in the opposite direction, we would come back nearly parallel to where we started off but further into the cups.  This contradicts that we started off at the deepest point of penetration.

These results are used in  \cite{CG} to show that for a hyperbolic $3$-manifold with finite volume there are finitely many branched surfaces that carry all closed index one embedded minimal surfaces.  The branched surfaces can be chosen to have curvature almost less than $-1$ and can be effectively constructed.    To quantify this we say that a branched surface is $\eta$-\emph{negatively curved} if all the sheets have sectional curvature $<\eta$.   Recall that minimal surfaces themselves in a hyperbolic manifold have curvature $\leq -1$ by the Gauss equation.   So the branched surfaces can be constructed to have almost the same upper curvature bound.  

\begin{Thm}  \label{t:t4}
If $N$ is a complete finite volume hyperbolic $3$-manifold and $\eta> -1$, then there exists finitely many effectively constructible $\eta$-negatively curved branched surfaces $B_1,\cdots,B_n$ such that any index $\leq 1$ closed embedded surface is carried by some $B_i$.  In particular any strongly irreducible Heegaard surface or is carried by one of these branched surfaces. 
\end{Thm}


For the proof we show first that there are finitely many branched surfaces that carry all closed stable embedded surfaces or more generally that carry all closed embedded minimal surfaces with $|A|^2\leq C$ for some large but fixed $C$.  The general case makes use of this, together with Theorem \ref{t:t2}, and a few additional arguments.  See \cite{CG} for more details.  

\section{Mean Convex Foliations}

Any closed orientable $3$-manifold with a strongly irreducible Heegaard splitting and bumpy metric has a natural singular mean convex foliation, \cite{CG}. This section provides the precise   statement and an idea of the proof.  

A \emph{mean convex foliation} in a Riemannian $n$-manifold with boundary is a smooth codimension one foliation, possibly with singularities of standard type, such that each leaf is closed and mean convex.

In a $3$-manifold a foliation with singularities of \emph{standard type} means that almost all leaves are completely smooth (i.e., without any singularities). In particular, any connected subset of the singular set is completely contained in a leaf. Moreover, the entire singular set is contained in finitely many (compact) embedded $C^1$ curves with cylinder singularities together with a countable set of spherical singularities. In higher dimensions there are direct generalizations of this.  

\begin{Thm}  \label{t:t5}
Any closed orientable bumpy Riemannian $3$-manifold M with a strongly irreducible Heegaard splitting, supports a mean convex foliation.
\end{Thm}

Bumpy is known to be a generic property.  In particular, it is known that there is a Baire category set of such metrics.  Bumpy is just the assertion that there does not exist an infinitesimal one-parameter family of closed minimal surfaces.  Said more precisely, bumpy means that the second variational operator of any closed minimal surface has trivial kernel so all weakly unstable closed minimal surfaces must be (strictly) unstable.   We will see next that as a consequence, for a bumpy metric, there is a mean convex foliation of a small neighborhood of a closed embedded min-max surface.   This follows from the fact that a small neighborhood of any unstable minimal surface can always be foliated whether or not the metric is bumpy.  To see this note, that if $\Sigma\subset N$ is a closed embedded unstable minimal surface, then the lowest eigenvalue $\lambda$ of the second variation operator $L$ is negative.  If $\phi$ is an eigenfunction for $L$ with eigenvalue $\lambda$, then $|\phi| > 0$  and so after possibly replacing $\phi$ by $-\phi$ we may assume $\phi > 0$. By the second variation formula if $\Sigma_s = F(x,s)$ is a variation of $\Sigma =\Sigma_0$ with $F_s\perp \Sigma_s$, $F_s(\cdot, 0) = \phi\,\nn$ , then $\frac{d}{ds}_{s=0} H_{\Sigma_s} =-L\,\phi=\lambda\,\phi<0$. 
Here $H$ is the mean curvature scalar in the direction of the unit normal $\nn$. It follows from this that for $s > 0$ sufficiently small the hypersurface $\Sigma_s$  lies on one side of $\Sigma$ and is mean convex.  If fact, it follows that, the surfaces $\Sigma_s$ for $|s|$ small gives a mean convex foliation of a neighborhood of $\Sigma$ with $\Sigma$ as one of the leaves.  

To extend the foliated neighborhood of an unstable minimal surface to an larger region we flow the surface by the mean curvature flow.  Mean curvature flow is the negative gradient flow of area, so any surface in a $3$-manifold flows through surfaces in the direction of steepest descent for area.  When the initial surface is mean convex, then the movement is monotone;  it moves only in one direction and keeps moving in that direction.  Thus, as it evolves it foliates a region.  As the surface evolve singularities can occur and the surface is not anymore smooth everywhere.  However, the flow is known to make sense past such singularities and the singularities are of standard type.  The movement only stops when the surface either collapses to a point, a simple closed curve, or another lower dimensional set, or the surface flows toward a stable minimal surface.  

The proof of Theorem \ref{t:t5} follows from a more general statement that roughly goes as follows.  For a closed orientable $3$-manifold, divide first the manifold into two along a closed embedded minimal surface obtained as the min-max surface of a sweep-out (obtained in Theorem \ref{pittsrubinstein}).  Assuming that the metric on the $3$-manifold is bumpy we can, as described above, foliate a small neighborhood of  either side of the min-max surface by strictly mean convex surfaces nearly parallel to the min-max surface.  Flow such a strictly mean convex surface by the mean curvature flow to get a possibly singular foliation with closed leaves that are mean convex.  If the flow does not sweep-out the ``entire side'' of the min-max surface, then it gets held up at a stable minimal surface.  If it sweeps out the entire side, then we have the desired singular foliation of that side.  However, when it gets held up, then we can obtain a new min-max surface on the side of the $3$-manifold that is bounded by the stable minimal surface.  Using Proposition \ref{stableboundary} we then repeat the process and either foliate a side or get held up by yet another stable minimal surface.  This process must stop after finitely many iterations and when it does, we have constructed a possibly singular foliation but with closed mean convex leaves.

  \begin{figure}[htbp]
\centering\includegraphics[totalheight=.58\textheight, width=.70\textwidth]{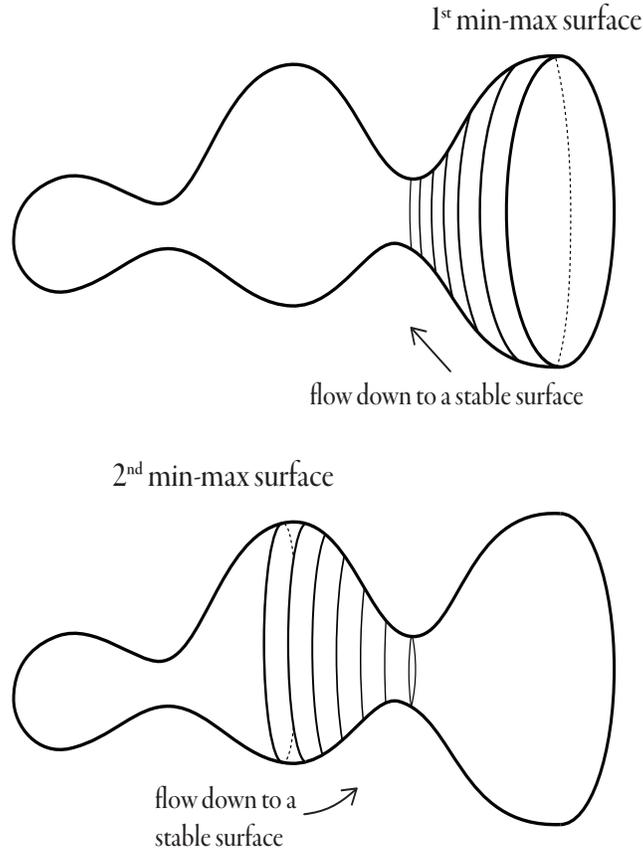}
\caption{A mean convex foliation of a $3$-manifold with three min-max surfaces.}   
  \end{figure}

\section{Polynomial vs Exponential Growth}  

By Theorems \ref{pitts} every strongly irreducible Heegaard surface or its 1-sided associate is isotopic to an index-$\le 1$ minimal surface.  In a hyperbolic 3-manifold such a surface has intrinsic sectional curvature $\le -1$.   Thus, such a Heegaard surface of genus-$g$, or its 1-sided associate could conceivably have an embedded disc with radius on the order of $\log(g)$.  This section shows that for a fixed hyperbolic 3-manifold $N$, there is an effectively computable uniform upper bound on the radius of an embedded disc on such a surface $H$.  In fact for fixed $L\in \BN$, there exists a uniform $r$ such that if $R\subset H$ is a connected surface containing a metric ball of radius $r$, then $\chi(R)\le -L$.  See Remarks \ref{pants} i).
\vskip 10pt
In what follows we make use of branched surfaces and laminations.  For precise definitions of related notions see \cite{O} and \S1 \cite{CG}.  Roughly speaking a \emph{measured lamination} is one for which one can associate a real number to arcs transverse to the lamination that are invariant under small perturbations through transverse arcs.  As an example, consider a foliation or lamination arising from a non singular 1-form.  Given a branched surface $B$, there is a natural way to take a regular neighborhood $N(B)$ whose boundary is a surface with corners.  Here $\partial N(B)=\partial_v N(B)\cup \partial_h N(B)$, where $\partial_v N(B)$ is a union of thin annuli parallel to the cusps of the branched surface and $\partial_h N(B)$ is a union of smooth surfaces with boundary locally parallel to the sheets of the branched surface.  A \emph{monogon} is a compact disc with a single cusp on its boundary.  It may arise as a properly embedded disc in the closed complement of a branched surface $B$. A monogon for $N(B)$ is a properly embedded disc in the closed complement of $N(B)$ which intersects $\partial N(B)$ in two arcs lying respectively in  $\partial_v N(B)$ and $\partial_h N(B)$.

\begin{definition}  Let $S$ be a Riemann manifold and $f:\BR\to \BR$.  We say that $S$ has growth at most $f$ if for each $x\in S$ and $r>0, \area(N_S(x,r))\le f(r)$.  If $f$ is a polynomial, then we say that $S$ has \emph{at most polynomial growth}, though usually just \emph{polynomial growth} for short.  If $e^{cr}<\area(N_S(x,r))$ for $r$ sufficiently large and $c>0$, then we say $S$ has \emph{exponential growth}.  \end{definition}   

\begin{remarks} \label{growth} i)  Since the area of a Euclidean (resp. hyperbolic) disc of radius $r$ is $\pi r^2$ (resp. $\pi sinh(r)^2$ ) their growth rates are respectively quadratic and exponential.  

	ii)  It can be deduced from \cite{CC1}, \cite{CC2} that there exists a foliation on $S\times I$, transverse to the $I$-fibers, where $S$ is the surface of genus-3, having leaves of exponential growth \emph{and} leaves of polynomial growth of all degrees.  Note that the closed surface of genus-3 carries the foliation. \end{remarks}
\vskip 10 pt

	On the other hand we have the following result, extending Plante's theorem \cite{Pl}, that leaves of measured foliations on closed manifolds have polynomial growth. 
	
\begin{theorem}  \label{plante} Let $B$ be a branched surface embedded in the Riemannian 3-manifold $M$.  There exists an effectively constructible polynomial $p(B)$, such that if $S$ is a leaf of a measured lamination carried by $B$, then the growth of $S$ is bounded by $p(B)$.\end{theorem}
	
\begin{corollary} Let M be a closed hyperbolic 3-manifold.  There exists an effectively computable polynomial $p(t)$ such that any index-$\le 1$ minimal surface has growth bounded by $p(t)$.\end{corollary}

\begin{proof}  By Theorem \ref{t:t4}  there is an effectively constructible set of branched surfaces that carry all such surfaces. Now apply Theorem \ref{plante}. \end{proof} 

These results together with Remarks \ref{growth} i) yield the following result  proved in \S 5 \cite{CG}.   By a \emph{regular splitting} of a branched surface we mean one that corresponds to splitting along the branch locus, as opposed to splitting to create a new complementary region. 

\begin{proposition}  \label{r splitting} Let $B$ be an $\eta$-negatively curved branched surface in the Riemannian 3-manifold M that fully carries a surface and let $r>0$.  Then there exist effectively constructible branched surfaces $B_1, \cdots, B_n$ obtained by regularly splitting $B$ such that every surface carried by $B$ is carried by some $B_i$ and each $B_i$ fully carries a surface.  Furthermore, if $E$ is  a subbranched surface of a regular splitting of some $B_i$, possibly $B_i$ itself,  then for each component $H$ of $\partial_h N(E)$ there exists $x\in H$ such that $N_H(x,r)\subset \inte(H)$.\end{proposition}

\begin{remarks} \label{pants} i) The crucial consequence of this result is that by \ref{growth} i)  given $N>0$, there exists an effectively computable $r>0$ such that $\chi(H)<-N$ for each $H$ as in the Proposition.

ii)  It follows that we can assume that the branch surfaces under consideration are \emph{horizontally large}, i.e. no component of $\partial_hN(B)$ is a disc or an annulus.\end{remarks}

\section{On the Classification of Heegaard Splittings I: Finiteness}

The long-standing classification problem is to exhibit for each closed 3-manifold a complete list, without duplication, of all its Heegaard splittings, up to isotopy.  In this section and the next two we survey the classification problem for irreducible splittings in irreducible 3-manifolds as well as the authors' solution for non Haken closed hyperbolic 3-manifolds.  This problem for  reducible splittings as well as reducible manifolds is also interesting. See \S 10 for statements of some results in that direction.

\vskip 10pt
Let M be a closed Haken 3-manifold.  In 1990 Klaus Johannson, announced the following result, Theorem 4 \cite{Jo1}, stated here for closed manifolds, asserting proof in the 446 page book \cite{Jo2}.

\begin{theorem}  Let $M$ be a closed Haken 3-manifold.  Modulo twisting along essential tori, the set of all genus-$g$ Heegaard splittings of $M$ is finite and constructible.\end{theorem}

In particular, if $M$ is also hyperbolic, there are only finitely many genus-$g$ Heegaard surfaces and they are constructible.  For Haken manifolds, this resolved a corrected form of Waldhausen's \cite{Wa1} conjecture: that a closed 3-manifold supports only finitely many Heegaard surfaces of a fixed genus.   

Now assume that $M$ is a closed non-Haken  3-manifold.  We have the following result, the last piece being done in \cite{CG}.

\begin{theorem}\label{complete} If $M$ is a closed non-Haken 3-manifold, then there exists an effective algorithm to produce a finite list of Heegaard surfaces which contains all the irreducible Heegaard splittings up to isotopy.\end{theorem}

\begin{remarks} i)  Note that this list may contain reducible splittings and duplicates, i.e. pairwise isotopic splittings.  

ii) The result for Seifert fibered spaces follows from the 1998 paper of Moriah and Schultens \cite{MS}.

iii) In 2006 Tao Li \cite{Li1}, \cite{Li2} proved that  $M$ has only finitely many isotopically distinct irreducible Heegaard splittings, thereby establishing a strong form of  Waldhausen's conjecture. In 2011 he showed that given $g>0$, there exists an effective algorithm to produce a finite list, possibly with duplication, of all the Heegaard surfaces of genus-$g$, up to isotopy.  

iv)  Since  by Perelman's geometrization theorem a closed non-Haken 3-manifold is either hyperbolic or Seifert fibered  it remained to effectively bound the genus of non Haken irreducible splittings of hyperbolic 3-manifolds.  This is done by the first two authors in the 2018 paper \cite{CG}.  \end{remarks}

\begin{theorem}  \label{main1} If $N$ is a closed non-Haken hyperbolic 3-manifold, then there exists an effectively computable $G(N)$ such that any irreducible Heegaard splitting of $N$ has genus bounded above by $G(N)$.  \end{theorem}

\begin{remarks}  \label{effective} i)  The paper \cite {CG} also effectively finds, for each $g$, a finite list of genus-$g$ splittings which contains the irreducible splittings.  It makes essential use of the negative curvature as well as deep results from minimal surface theory.  Tao Li \cite{Li2} on the other hand only starts with a triangulation. With bare hands, armed only with normal surface theory, he proves his result. 

ii) An \emph{effectively computable algorithm}, is one that produces an output within some function of the input data.  In Theorem \ref{main1}  the input  is a triangulation by hyperbolic simplices, each of which has uniformly bounded dihedral angles and edge lengths.  Such a triangulation exists by \cite{Br}.   

iii) An effective algorithm starting from a combinatorial triangulation would follow from \S 7 \cite{CG} and a positive solution to Conjecture \ref{branched surface}.\end{remarks}
\noindent\emph{Idea of the proof of Theorem \ref{main1}}.  To start with \cite{CaGo1} implies that the irreducible splittings are also strongly irreducible.  By \cite{KLS} a strongly irreducible Heegaard surface $S$ is isotopic to either an index-$\le 1$ minimal surface or by attaching an unknotted tube between the sheets of the double cover of a non orientable index-0 surface.  By Theorem \ref{t:t4} these index-$\le 1$ surfaces are carried by finitely many $\eta$-negatively curved branched surfaces $B_1, \cdots, B_n$.  By Theorem \ref{r splitting} and Remark \ref{pants} we can assume that each sub branched surface $B$ of each $B_i$ has the property that  $\partial_h(N(B_i))$ contains a $\pi_1$-injective pair of pants.  Let $F_1, \cdots, F_q$ represent the fundamental solutions to the normal surface equations of $B_i$. The goal is to find an $N_i\in \BN$ such that if $S=n_1F_1+\cdots+n_qF_q$, some $n_j\ge N_i$ and  $S$ is a Heegaard surface then $S$ is weakly reducible. Suppose that $0<n_1\le n_2\le \cdots\le n_q$.   It turns out that we can then assume $n_1=1$ or we readily find a weak reduction.  Also we can assume that for some very large $N_i$, if $j$ is the smallest value with $n_j\ge N_i$, then $n_{j-1}/n_j$ is very small.  Let $S_2 = n_j F_j+ \cdots n_q F_q$. If $B$ is the subbranched surface that fully carries $S_2$, then either $B$ is incompressible and hence $N$ is Haken or $\partial_h N(B)$ is compressible.  Since $N$ is non Haken, the former case does not occur.  Let $D$ denote a compressing disc.  We can assume that $P\subset \partial_h N(B)$ is an essential pair of pants where one component of $\partial P = \partial D$.  For simplicity assume that it projects to an embedded surface $P'$ in $B$.  It follows that $S_2$ contains at least $N_i$ parallel copies of $P$ that project to $P'$.   In the most interesting situation, a compressing disc $D$ for $\partial_h N(B)$ intersects $B_i$ in a single arc $\alpha$ which decomposes $D$ into two monogons.  Also $S$ passes through the sector $\sigma$ of $B_i$ which contains $\alpha$ exactly once.     These monogons then extend to isotopically disjoint compressing discs for $S$, one on each side of $S$.  Each such disc consists of two copies of a monogon and a strip that lies in the interstitial bundle of $S$ and which penetrates only a uniformly bounded amount. The projection of this strip to $B_i$ is a properly immersed arc in $P'$.  \vskip8pt

\begin{remark} The ideas of using branched and normal surface theory are already in  \cite{Li1} and \cite{Li2} as is the idea of extending monogons to find disjoint compressing discs.  The hyperbolicity enabled us to effectively find finitely many $\eta$-negatively curved branched surfaces which in combination with Theorem \ref{plante} enabled us to split a controlled amount to find ones with large horizontal boundaries, e.g. the $P\subset \partial_h N(B)$.  A detailed understanding of $S$ near $N(P')$ enabled us to effectively find the desired strips that built the weakly compressing discs.  \end{remark}

\section{On the Classification of Heegaard Splittings II: \newline The Thick Isotopy Lemma}

The goal of this section is to outline a proof of the following result.

\begin{theorem} \label{main} Let $N$ be a closed non-Haken hyperbolic 3-manifold. There exists an effectively constructible set $S_0, S_1, \cdots, S_n$ such that if $S$ is an irreducible Heegaard splitting, then $S$ is isotopic to exactly one $S_i$.\end{theorem}

By Theorem \ref{complete} we can effectively construct a set of Heegaard surfaces containing all the irreducible ones.  To prove Theorem \ref{main} we need to effectively weed out duplications as well as reducible splittings from this set.    

Let us first focus on the elimination of duplicate irreducible splittings. The main point is that when two strongly irreducible surfaces are isotopic in a hyperbolic manifold, we can apply the following \emph{Thick isotopy Lemma I} to find a path connecting them with controlled geometry in the sense that the areas of surfaces in the isotopy are bounded from above by a computable amount and the surfaces never get \emph{too thin} to either side.  To formalize the notion of \emph{thinness}, let us say a surface $\Sigma$ embedded in a three-manifold $N$ is $\delta$-incompressible if there exists no essential simple closed curve of diameter (in $N$) at most $\delta$ that bounds a disk in a complementary region. 
A finite net in the family of such surfaces is effectively constructible and gives rise to a graph $\mG$ whose vertices are elements of the net and where two vertices are connected by an edge if they correspond to \emph{close} surfaces.  
An isotopy with controlled geometry gives rise to a path in this graph. Thus, two strongly irreducible splittings are isotopic if and only if they lie in the same component of the graph.   See \cite{CGK} \S3 for details.  For technical reasons the proof is conducted in the PL category.  It uses the notions of \emph{crudely normal surface} and \emph{pinched isotopy} which may be of independent interest.

\begin{lemma}[Thick Isotopy Lemma I \cite{CGK}]\label{thin1}
Let $N$ be a hyperbolic non-Haken three-manifold of injectivity radius $\delta_0$.  Suppose $\Sigma_0$ and $\Sigma_1$ are $8\delta$-incompressible isotopic Heegaard surfaces where $\delta<\delta_0/16$.  Then there exists an isotopy $\Sigma_t$ joining $\Sigma_0$ to $\Sigma_1$ so that each surface $\Sigma_t$ is $\delta$-incompressible and the area of all intermediate surfaces $\Sigma_t$ is bounded by a computable constant $C$. 
\end{lemma}

To start with, given the Heegaard surfaces $\Sigma_i$, $i=0,1$, we can explicitly construct, e.g. by Haken's algorithm \cite{Ha2}, a Heegaard foliation $\mF_i$ with $\Sigma_i$ as a leaf and hence compute an upper bound $C_1$ for the area of any surface in either $\mF_0$ or $\mF_1$.  Since $\Sigma_0, \Sigma_1$ are isotopic Heegaard surfaces, there exists an extension to an isotopy of $\mF_0$ to $\mF_1$.  Thus, there exists \emph{some} 2-parameter sweepout $\Sigma_{s,t}$ parametrized by $X=I\times I$ where $\Sigma_{i,t}, t\in I$ corresponds to $\mF_i$  for $i=0,1$ and each $\Sigma_{t,0}$, $\Sigma_{t,1}$ is a 1-complex.  

The areas of the surfaces in the sweepout $\Sigma_{s,t}$ \emph{a priori} may be arbitrarily large.  Using min-max theory we claim that there exists a sweepout parametrized by X with the same boundary values as $\Sigma_{s,t}$ such that:
\vskip 10pt
(*) for each $s,t$, $\mbox{area}(\Sigma_{s,t})< C=\max\{C_1+1, 4\pi (g-1)+1\}$, where $g=\genus(\Sigma_i)$.
\vskip 10pt


To prove (*), we will also need the following fact (which follows from the Gauss equation and Gauss-Bonnet formula):
\begin{lemma}[Area bounds for minimal surfaces]\label{hyperbolicareabound}
Let $\Gamma$ be a genus $g$ minimal surface immersed in a hyperbolic three-manifold.  Then 
\begin{equation}\label{areabounds}
\mbox{area}(\Gamma)<4\pi(g-1).   
\end{equation}
\end{lemma}

The genus bounds together with Lemma \ref{hyperbolicareabound} imply:
\begin{lemma}\label{bounds}
The width of a non-trivial $k$-parameter family of genus $g$ surfaces in a hyperbolic three-manifold is at most $4\pi(g-1)$.  
\end{lemma}
\emph{Proof:}
By the min-max theorem, we obtain the existence of a collection of minimal surfaces $\Gamma_1,..., \Gamma_k$ as well as positive integers $n_1...,n_k$ so that $W=\sum n_i\mbox{area}(\Gamma_i)$.  Let us assume for simplicity the $\Gamma_i$ are all orientable.   Then we have:
\begin{equation}
W=\sum n_i \mbox{area}(\Gamma_i) <  4\pi \sum n_i (\mbox{genus}(\Gamma_i)-1) \leq-4\pi+4\pi \sum n_i (\mbox{genus}(\Gamma_i)) \leq   4\pi  (g-1).    \end{equation}
The first inequality is from Lemma \ref{hyperbolicareabound} and the last inequality follows from the Genus Bounds \eqref{genusbounds}. 
\qed
\\
\\
Roughly speaking, Lemma \ref{bounds} asserts that no matter how complicated topologically a genus $g$ sweep-out might be in a hyperbolic manifold, by pulling the entire family as tightly as possible, the areas of the surfaces can be controlled just in terms of the genus $g$.   
\\
\\
\emph{Proof of (*):}

We consider the saturation $\Lambda$ of the sweepout $\Sigma_{s,t}$ and the corresponding min-max value $W_\Lambda$.  There are two cases: either $W_\Lambda > C_1$ or else $W_\Lambda \leq  C_1$.  In the second case, from the definition of width we have sweepouts with areas satisfying (*).  In the first case, the family is non-trivial and Lemma \ref{bounds} together with the definition of width then gives (*).\qed



\vskip 10pt
We now continue the proof of Lemma \ref{thin1}.  Suppose the sweepout obtained from (*) has been parametrized so that for $i=0,1, \Sigma_i = \Sigma_{i,1/2}$.  Thus, any properly embedded path $\alpha(t)$ in $X$ from $(0,1/2)$ to $(1,1/2)$ gives rise to an isotopy defined by $\Sigma_t=\Sigma_{\alpha(t)}$.  By construction $\area(\Sigma_t)<C$.  To complete the proof of Lemma \ref{thin1} we need to show that there exists a path $\alpha(t)$ so that each $\Sigma_t$ is $\delta$-incompressible. Suppressing the details, the idea of the proof is contained in the following figure.  The figure shows that the failure of finding the desired path implies that $\Sigma_0$ is weakly reducible and hence reducible \cite{CaGo1}, since $N$ is non-Haken.  (In reality, \cite{CGK} the green region $G$ is compact with piecewise smooth boundary and contains all the parameters $(r,s)$ such that $\Sigma_{r,s}$ has a $\le\delta$ short compression to the $H_0$ side and if $(r,s)\in G$, then $\Sigma_{r,s}$ has some $\le 1.5\delta$ compression to the $H_0$ side.   Also $(0\times [1/2,1]\cup 1\times [1/2,1])\cap G=\emptyset$.  Analogous statements hold for the red region $R$, with $H_0$ replaced by $H_1$ and $(0\times [0,1/2]\cup 1\times [0,1/2])\cap R=\emptyset$. If $(r,s)\in G\cap R$, then $\Sigma_{r,s}$ is $1.5\delta$-bicompressible and hence the corresponding Heegaard splitting is weakly reducible as we discuss below.)

\begin{figure}[ht]
\includegraphics[scale=0.60]{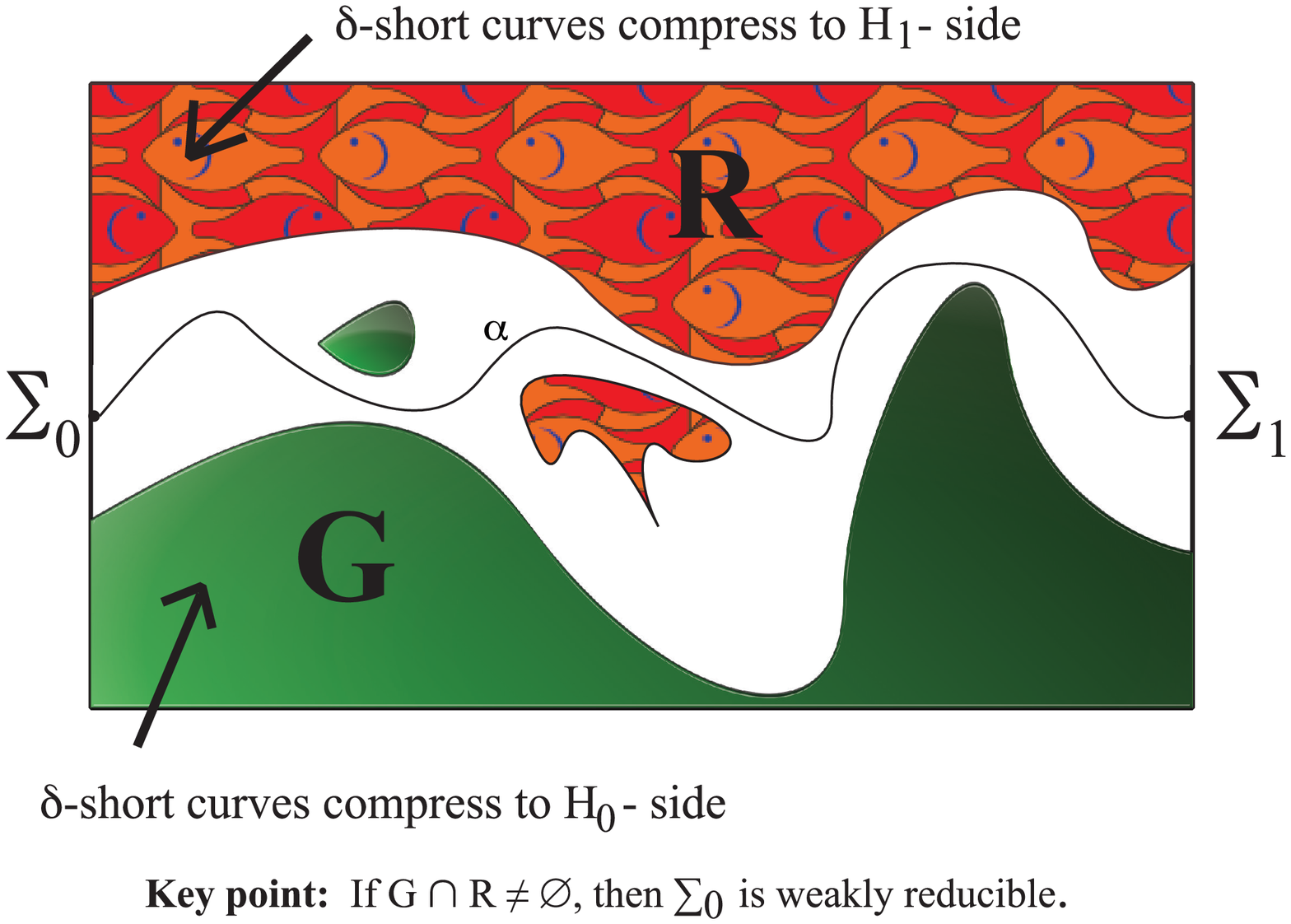}
\end{figure}

Let us now focus on the elimination of reducible splittings. Here we introduce the notion of $\Sigma$ being $\eta$-bicompressible. This means that there exist essential simple closed curves of diameter $\le \eta$ which respectively compress to distinct sides of $\Sigma$. A simple fact is that if $2\eta <$  the injectivity radius of $N$, and N is not the 3-sphere, then an $\eta$-bicompressible Heegaard surface is weakly reducible.  The key geometric result needed is Lemma \ref{thin2}, an analogue of Lemma \ref{thin1} which roughly says that if $\Sigma$ is weakly reducible, then for $\delta$ sufficiently small, it is isotopic to an $4\delta$-bicompressible surface through $\delta$-incompressible surfaces whose areas are uniformly bounded above.  Using the same graph $\mG$ we see that $\Sigma$ is reducible if and only if it is in the same component as an $\eta$-bicompressible one.  See \cite{CGK} \S3 for details.  In both cases, i.e. the weeding out of duplications and reducibles, the path in $\mG$ corresponds to a pinched crudely normal isotopy with respect to a triangulation $\Delta_3$ such that the weights of the interpolating surfaces remain uniformly bounded above.  That we can uniformly bound the weights follows from the area bound in Lemmas \ref{thin1}, \ref{thin2}.  The $\delta$-incompressibility condition allows us to work in the category of crudely normal and crudely almost normal surfaces.  The number of such surfaces of uniformly bounded weight is finite, hence the finiteness of $\mG$.  By construction $\Delta_3$ is a subcomplex of a much coarser $\Delta_2$.  It is at the level of $\Delta_2$ that bicompressibility is effectively detected.  Very heuristically, from the eyes of $\Delta_3$, a path in $\mG$ looks like an isotopy of an incompressible surface, while the vision of $\Delta_2$ is sufficiently broad to detect weak reducibility.

\begin{lemma}[Thick Isotopy Lemma II \cite{CGK}]\label{thin2}  Let $N$ be a closed non-Haken hyperbolic 3-manifold with injectivity radius $\delta_0$.   If $\Sigma_0$ is weakly reducible and $8\delta$-locally incompressible, then there exists an effectively computable $C$ and an isotopy $\Sigma_t$ from $\Sigma_0$ to a $\Sigma_{t_1}$ such that for each $t\le t_1$, $\area(\Sigma_t)<C$ and $\Sigma_t$ is $\delta$-incompressible.  Finally $\Sigma_{t_1}$ is $4\delta$-bicompressible.\end{lemma}  

This means that at the scale of $\delta$, each $\Sigma_t$ looks incompressible, while at the scale of $4\delta$, $\Sigma_{t_1}$ is weakly reducible.  Note that the compressing discs need not have small diameter.

Here is the idea of the proof of Lemma \ref{thin2}.  As before we let $\mF_0$ be a Heegaard foliation extending $\Sigma_0$ and $C_1$ the maximal area of its leaves.  If $\Sigma_0$ is weakly reducible, then it is reducible, so it is isotopic to $\Sigma_1$ with Heegaard foliation $\mF_1$  such that each leaf of $\mF_1$ is $ \delta$-bicompressible.  Indeed $\Sigma_1$ can be taken to be a stabilization of a strongly irreducible splitting and  so $\mF_1$ can be constructed so that the curves in the trivial handle part of the stabilization  have diameter $<\delta$.   Further using Theorem 
\ref{t:t5} or Theorem \ref{pittsrubinstein} and two applications of Lemma \ref{bounds}, $\mF_1$ can be also constructed so  such that the area of each leaf is $<4\pi ((g-1)-1)+1<4\pi(g-1)$.    Thus there exists a 2-parameter sweepout parametrized by $X=I\times I$ such that for $i=0,1$, \ $ \Sigma_{i,t}$ are leaves of $\mF_0$ and $\mF_1$ and for $t\in [0,1],\ i\in \{0,1\}, \Sigma_{t,i}$ is a 1-complex.  Thus for $x\in \partial X$, $\area(\Sigma_x)< C=\max\{C_1+1,4\pi(g-1)+1\}$.  As in the proof of Lemma \ref{thin1}, there exists a sweepout $\Sigma_{s,t}$ parametrized by $X$ taking on the same boundary values as the original one, such that for each $x\in X, \area(\Sigma_x)< C$.  

Now parametrize $\mF_0$ so that $\Sigma_0=\Sigma_{0,1/2}$.  To complete the proof it suffices to find a smooth path $\alpha(t)\subset X, t\in [0, t_1]$ such that $\alpha_0=(0,1/2),\alpha\cap \partial X=\alpha_0$,  each $\Sigma_{\alpha(t)}$ is $\delta$-incompressible and $\Sigma_{\alpha(t_1)}$ is $4\delta$-bicompressible.  The idea of the proof that such an $\alpha(t)$ exists, is contained in the following figure.  See \S 3 \cite {CGK} for the details.
\begin{figure}[ht]
\includegraphics[scale=0.60]{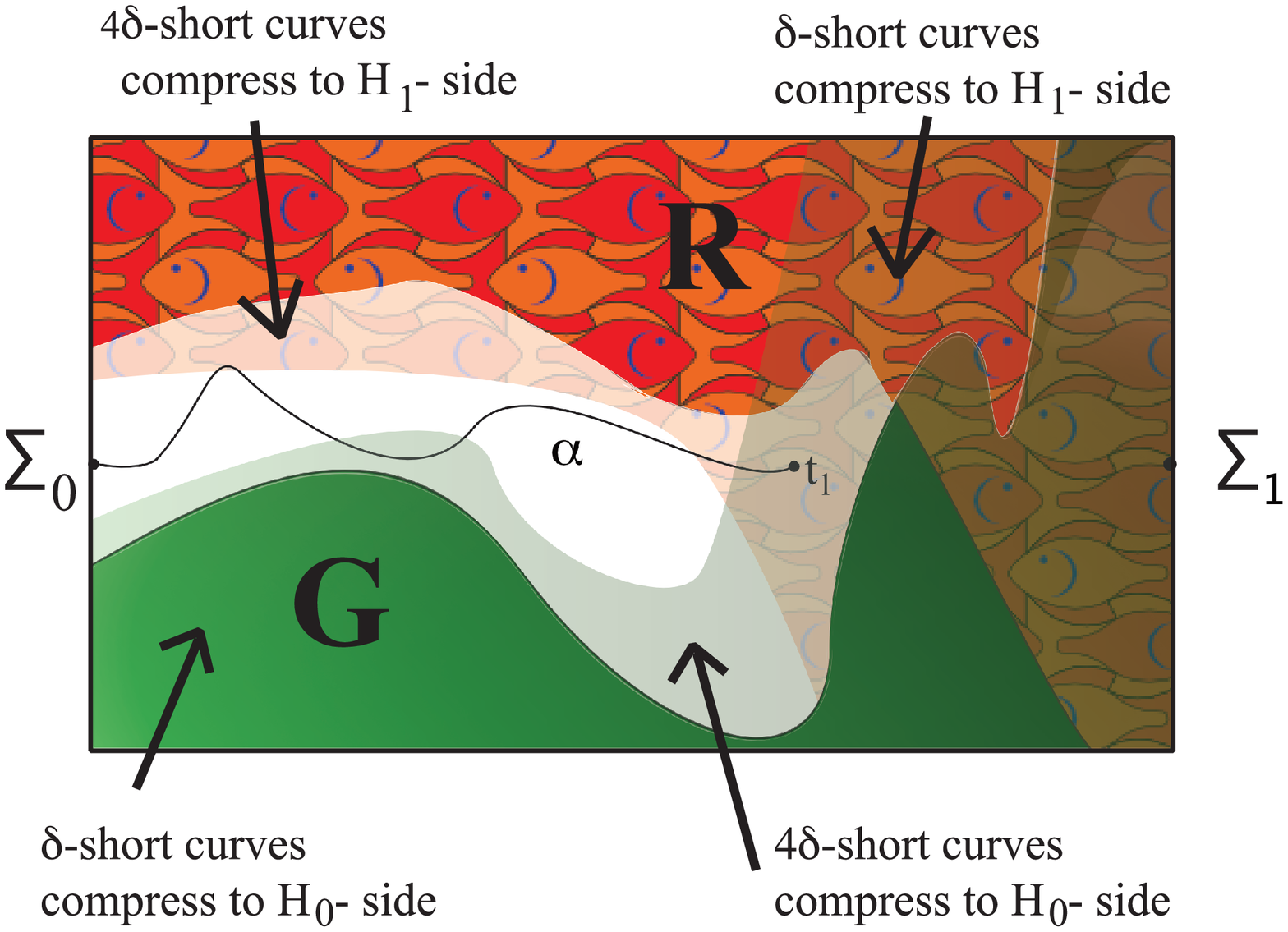}
\end{figure}

\section{Problems}

\subsection{The Heegaard tree}  

Associate to a 3-manifold $M$ a directed graph $\mathcal{T'}$ whose vertices are isotopy classes of Heegaard splittings of $M$ and an edge points from $v$ to $v'$ if a splitting representing $v$ is a stabilization of one representing $v'$.  The Reidemeister - Singer theorem \cite{Re}, \cite{Si} implies that any two Heegaard splittings have a common stabilization. Since stabilization is unique up to isotopy, $\mathcal{T'}$ is a tree.  We define the \emph{Heegaard tree} of $M$ to be the minimal subtree $\mathcal{T}(M)$ that contains all the irreducible splittings.  

\begin{problem} Give an effective algorithm to construct the Heegaard tree for a closed irreducible non-Haken 3-manifold $M$. In particular find an effective algorithm to compute how many stabilizations are needed to make distinct Heegaard splittings isotopic. \end{problem}

\begin{remarks} i)  Rubinstein and Scharlemann \cite{RS3} show that 5p+7q-9 stabilizations suffice where q (resp. p) is the larger (resp. smaller) genus of the Heegaard splittings.  See \cite {Tak} for an approach using singularity theory. Combining these results with \cite{Li2} and \cite{CG} it follows that $\mathcal{T}(M)$ is finite and the number of vertices is bounded by an effectively computable function.

ii) Johannson \cite{Jo2} gave a polynomial upper bound on the number of stabilizations needed for splittings of Haken manifolds to become equivalent.
 
iii) More than one stabilization may be needed \cite{HTT}, \cite{Ba4}, \cite{Jon2}.  

iv) For reducible 3-manifolds there are the foundational results of Bachman \cite{Ba2} and Qiu-Scharlemann \cite{QS} solving what was known as the Gordon conjecture.  They showed that if the closed 3-manifold $M$ is the connect sum of $M_1$ and $M_2$, then the sum of unstabilized splittings in $M_1$ and $M_2$ is unstabilized.  Bachman \cite{Ba2} further showed that an unstabilized splitting in $M$ has a unique expression as the connect sum of Heegaard splittings of prime 3-manifolds. \end{remarks}
\vskip 10pt

\subsection{Effective splitting of branched surfaces}
\begin{definition}  We say that the branched surface $B\subset M$ is \emph{quasi-hyperbolic} if it does not carry any sphere or torus, but fully carries a surface.  A \emph{regular splitting} of $B$ is one which opens up $B$ along its branch locus. \end{definition}

Note that fully carrying $\eta$-negatively curved branched surfaces are quasi-hyperbolic.

\begin{conjecture} \label{branched surface} Let $B$ be a branched surface in the compact triangulated atoroidal irreducible 3-manifold $N$.  Then there exist finitely many effectively constructible  quasi-hyperbolic branched surfaces $B_1, \cdots, B_n$ such that each $B_i$ is the result of passing to a subbranched surface of some regular splitting of $B$ and every strongly irreducible or incompressible surface carried by $B$ is carried by some $B_i$.\end{conjecture}

\begin{remarks}   i) In 2007 Tao Li proved the non effective version of this conjecture \cite{Li1}.  There the passage from $B$ to $B_1, \cdots, B_n$ is obtained via a compactness argument.  See Proposition 8.1 \cite{Li1}. 

ii) \S 7 \cite{CG} shows that a proof of this  conjecture would complete an effective  proof of Theorem \ref{main} starting only with a combinatorial triangulation.\end{remarks}
\vskip 10pt
\subsection{Index of common stabilizations}
The following is a special case of a conjecture of David Bachman \cite{Ba3}.

\begin{conjecture} If $M$ is a non-Haken Riemannian 3-manifold, then the minimal common stabilization surface $S$ of distinct irreducible Heegaard splittings is isotopic to a surface of index-$\le 2$.\end{conjecture}

\begin{remarks} i) In \cite{Ba3} Bachman defined the topological index of certain surfaces $H\neq T^2, S^2$ as follows.  Define the \emph{disc complex} $\Delta(H)$, the complex whose vertices are isotopy classes of embedded discs in $M$ that intersect $H$ exactly along their boundaries and in essential curves.  Its $m$-simplices are $m+1$-tuples of distinct vertices representable by pairwise disjoint discs.  Define $H$ to have topological index $0$ if $\Delta(H)=\emptyset$, and topologically of index-$k$, for $k\ge 1, $ if $ \pi_{k-1}(\Delta(H))$ is the first non trivial homotopy group of $\Delta(H)$.  He conjectured that a surface of topological index-k is isotopic to a surface of index-$\le k$.  For $k=0$, this is true by \cite{FHS}.  A Heegaard surface of topological index-1 is strongly irreducible \cite{CaGo1} and hence isotopic to an index-$\le 1$ surface by \cite{KLS}.  Note that by McCullough the disc complex of a handlebody is contractible \cite{Mc}, hence does not have a defined topological index.  

ii) In 2010 Daniel Appel \cite{Ap} proved that the Heegaard surface of genus $g\ge 2$ in the 3-sphere has topological index $2g-1$.  Furthermore these are the only connected surfaces in the 3-sphere with non trivial topological index.  On the other hand the minimal index for a Heegaard surface of genus-2 is at least $6$ \cite{Ur}.   Appel's results disprove Conjectures 5.6 and 5.9 of \cite{Ba3} and  give positive solutions to Questions 5.3, that there is a non Haken 3-manifold with surfaces of topological index $\ge 3$ and Question 5.5, that there is a 3-manifold with surfaces of arbitrarily high topological index.  Can Bachman's conjectures be modified to take   account of the special nature of $S^3$?\end{remarks}
\vskip 10 pt
\subsection{The Goeritz conjecture}
\begin{definition} Define the genus-$g$ Goeritz group $\mH_g$ as the group of isotopy classes of orientation preserving diffeomorphisms of the 3-sphere that leave the standard genus-g Heegaard splitting invariant.\end{definition} 

\begin{problem} Is  $\mH_g$ finitely generated and if so find a set of generators? \cite{Gor2}, \cite{Po}, \cite{Sc3}. \end{problem}

\begin{remarks} i) For $g=0, 1$ this is the trivial group.  L. Goeritz \cite{Gor2} answered this question for $g=2$.  A modern proof can be found in \cite{Sc3}.

ii) J. Powell \cite{Po} proposed a set of generators for the general $\mH_g$. His argument that they sufficed, had a gap.   See \cite{Sc3}.

iii) M. Freedman and M. Scharlemann recently proved that Powell's generators suffice for $\mH_3$, \cite{FS}.  E. Akbas \cite{Ak} and S. Cho \cite{Cho} showed that $\mH_2$ is finitely presented.

iv) Results on the analogous Goeritz group in 3-manifolds can be found in \cite{JM}.\end{remarks}

\vskip 10 pt
\subsection{More problems}
For other problems and questions see the Appendix of this paper, \cite{Li5}, \cite{So}, \cite{Go1} and \cite{BDS}.

\section{Appendix: On the Geometry of Handlebodies, David Gabai and Tobias Holck Colding}  

Motivated by the following remarkable result, this appendix speculates on the geometry of ultra large volume handlebodies with Cheeger constant uniformly bounded below. 

\begin{theorem} \label{tau} (Long - Lubotzky - Reid \cite{LLR}, Bourgain - Gamburd \cite{BoGa})
Let $N$ be a  closed, connected hyperbolic 3-manifold.  Then there exists an infinite tower $\cdots \to N_2\to N_1\to N_0=N$ such that 

i) For all $i$, $Ch(N_i)>c_1>0$,

ii) $\injrad(N_i)\to \infty$ and

iii) $N_j$ is regular cover of $N_i$ if $j>i$.

Here $Ch(N_i)$ denotes the  Cheeger constant of $N_i$.\end{theorem}

In addition we have 

\begin{theorem} (Lackenby \cite{La2})

iv) For all $i$, $ g(N_i)/\vol(N_i)>c_2>0$, where $g(N_i)$ denotes the Heegaard genus of $N_i$.\end{theorem}

\begin{remarks}  1)  Recall that the \emph{Cheeger constant} $(M) = \inf \{\area(A_1\cap A_2)/\min\{\vol(A_1), \vol(A_2)\}\}$ where $A_1\cup A_2=M$ and $A_1\cap A_2=\partial A_1=\partial A_2$.

2) Clozel \cite{Cl} proved that congruence covers of a given closed arithmetic 3-manifold have a uniform lower bound for their Cheeger constants.\end{remarks}

\begin{question} \label{vhc} Given a sequence $\{N_i\}$ as above, is $N_j$ Haken for $j$ sufficiently large?  \end{question}

\begin{remarks} \label{goal} 1)  We investigated this question in an attempt to address the now resolved virtual Haken conjecture due to I. Agol \cite{Agol}, D. Wise and his collaborators e.g. \cite{Wi}, and J. Kahn - V. Markovic \cite{KM}.  As far as we know,  Question \ref{vhc} is still open.     

2)  In general either the $N_i$'s are eventually Haken or by \cite{CaGo1} there is a sequence $(H^i_0, H^i_1)$ of minimal genus strongly irreducible Heegaard splittings of the $N_i$'s.  By the Pitts - Rubinstein conjecture \ref{pitts}, \cite{KLS} we can assume that $S^i=H^i_0\cap H^i_1$ is minimal of index-$\le 1$.  By i) and iv)  both $Ch(N_i)$ and $g(N_i)/\vol(N_i)$ are uniformly bounded below. 

3) This leads to the following question which we think is of independent interest and worthy of further study.  Papers \cite{CG} and \cite{CGK} had its origins in efforts to understand this question.\end{remarks}

\begin{question} \label{handlebody geometry}  Let $(H_0, H_1)$ be a strongly irreducible Heegaard splitting of the closed hyperbolic 3-manifold $N$ with $H_0\cap H_1=S$.  What is the \emph{geometry} of $H_0$ and $H_1$ if $S$ is a minimal surface of index-$\le 1$ and $Ch(N)$  and $g(N)/\vol(N)$ are uniformly bounded below where $\vol(N)>>0$.\end{question}  

We now briefly describe four \emph{types} of handlebodies and for the larger of $H_0$ and $H_1$ offer a potential model.  The first three describe known constructions.    \vskip 10 pt

\noindent\textbf{Neighborhoods of 1-complexes.}  This is the way we usually first think of handlebodies.  Here the neighborhood is not meant to be too large, so that each element of a \emph{complete set} of compressing discs, i.e. a set that cuts the handlebody into a ball,  has both small area and boundary length.  
\vskip 10 pt

\noindent\textbf{Generalized Cannon - Thurston handlebodies.}    i) Let $S$ be a closed surface of genus $g$ with two transverse binding measured foliations (or geodesic laminations) $(\mu_1, dx)$ and $(\mu_2, dy)$.  \emph{Binding} means that every essential simple closed curve $\gamma$  has positive measure with respect to one of $\mu_1$ or $\mu_2$, where the measure of a curve is calculated as the infimum over all curves in its homotopy class.   This defines a natural pseudo-metric $ds_0^2=dx^2+dy^2$ on $S$.  Given $\lambda>0$ and $k>1$, J. Cannon and W. Thurston \cite{CT} define an infinitesimal pseudo-metric $\rho$ on $S\times (-\infty, \infty)$ by $ds^2=k^{2t} dx^2 + k^{-2t}dy^2 +\lambda^2 dt^2$, where $t$ comes from the second factor.  They showed that if $\mu_1, \mu_2$ are the invariant foliations of a pseudo-Anosov mapping $f$ with stretch factor $k$, then $\rho$ is quasi-comparable with the hyperbolic metric $\hat\rho_f$ on $S\times (-\infty, \infty)$.  The latter is the pull back to the infinite cyclic cover, of  the hyperbolic metric $\rho_f$ on the closed manifold $M_f$ that fibers over $S^1$ with monodromy $f$.  Here the infinite cyclic cover $S\times (-\infty, \infty)$ is parametrized so that $S\times [i,j]$ is the manifold obtained by gluing $j-i$ fundamental domains.  See \cite{CT} for a detailed analysis of  $\rho$  on $S\times (-\infty, \infty)$ and  its pull back to $\BH^2\times (-\infty, \infty)$.

Given $g$ pairwise disjoint simple closed curves in $S\times 0$ cutting it into a planar surface, and $i\in \BN$ construct a handlebody $H_i$ by restricting the $\rho$ metric to $S\times [0,i]$ and attaching 2-handles to the $S\times 0$-side along these curves and capping off the resulting 2-sphere with a 3-ball.   We call such a handlebody a \emph{Cannon - Thurston Handlebody}.  
\vskip10pt

\noindent ii)  Let $f:S\to S$ be a pseudo-Anosov mapping of a closed surface of genus-$g$ and $M_f$ the mapping torus.  Let $\rho_f$ be the hyperbolic metric on $M_f$ and  $\epsilon >0$.  Let $H_i$ be the handlebody of genus-$g$ parametrized as $S\times [0,i]$ with a handlebody attached to the $S\times 0$-side.  H. Namazi and J. Souto \cite{NaSo} showed that if $i$ is sufficiently large, then one can construct a hyperbolic structure on $H_i$ whose restriction to $S\times [0,i]$ is $\epsilon$-close to the restricted metric $\hat\rho_f$ defined in i).  J. Hass, A. Thompson and W. Thurston \cite{HTT} showed  by modifying the Namazi - Souto construction, that for $j$ sufficiently large, there exists a Riemannian metric $\rho_j$ on  $H_j$ with sectional curvatures between $-1-\epsilon$ and $-1+\epsilon$ such $\rho_j$ coincides with $\rho_f$ near $S\times j$ and hence, for $k>j$ one can construct $\rho_k$ on $H_k$ which coincides with $\rho_j$ on $H_j$ and $\hat\rho_f$ on $S\times [j,k]$.
\vskip10pt

iii) The following is a very rough approximation of a special case of a construction of \cite{BMNS1} due to J. Brock, Y. Minsky, H. Namazi and J. Souto.   Construct a handlebody $K$ by gluing together finitely many compression bodies, where the $-$-side of one is glued to the $+$-side of another.  Geometrically, there is a long product region between any two compression bodies whose fiberwise metric is given by a segment of a thick Teichmuller geodesic.  The compression bodies themselves have fixed Riemannian metrics.  
\vskip 8pt

\begin{remark}  Note that i) yields handlebodies with either small Cheeger constant $C$ or compressing discs whose boundary lengths are of the order $\log(g)$.  Here we assume that short curves are used on $S\times 0$.  In ii) and iii) either there is a small Cheeger constant or the geometric complexity of the manifold is mostly in the unknown compression bodies.  For us, these are the regions of interest.  \end{remark} 

\noindent\textbf{Neighborhoods of Gropes.}  Let $H$ be a handlebody inductively constructed as follows.  Start with $G_1= S\times I$, where $S$ is a compact surface with non empty connected boundary.   Construct the second stage  $G_2$ as follows.  Glue $T_1\times I, \cdots, T_n\times I$ to $S\times 0\cup S\times 1$ where $T_1, \cdots, T_n$ are compact surfaces with non empty connected boundary and the $(\partial T_i)\times I$'s are glued to pairwise disjoint neighborhoods of  simple closed curves $\gamma_1,\cdots, \gamma_m$ on $S\times \{0,1\}$.  We require that if $P_2$ is the projection to $S$ of the union of the  $\gamma_i$'s, then $S$ deformation retracts to a 1-complex that includes $P_2$.   In a similar manner for some $k\in \BN$ construct $G_3, G_4,\cdots G_k$ where each $G_p$ is obtained from $G_{p-1}$ by attaching thickened surfaces to $G_{p-1}\setminus G_{p-2}$ as above.    It is an exercise to show that $G_k$ is a handlebody and to find systems of compressing discs.  $G_k$ naturally deformation retracts to a 2-complex called a \emph{grope}.  See \cite{FQ}. 

\begin{remark}  Suppose that for each  $T_i\times I$ of $G_p\setminus G_{p-1}, \length(\partial T_i)$ and the total length of a set of arcs cutting $T_i$ into a disc are uniformly bounded, as is the length of the $I$-fibers.   Since for $i<j$ annuli of uniformly bounded area separate off components of $G_k\setminus G_i$, neighborhoods of gropes generally have very small Cheeger constant.  Also, $G_k$ has some compressing discs with  boundaries of uniformly bounded length, but a complete set of compressing discs contains discs whose lengths grow exponentially in $k$.  For example, let $T$ denote the torus with an open disc removed.  Let $G_1=T\times I$ and $G_{i+1}$ be obtained from $G_{i}$ by attaching a single $T\times I$ to the $i$'th stage.      Suppose that $\vol(T\times I)=1$ and the length of $\partial T$ and the lengths of two disjoint essential arcs in $T$ have length 1.  Then $\vol(G_k)=k$ and $G_k$ has a compressing disc of approximate boundary length 2, but a complete system requires a disc of boundary length approximately $2^k$.\end{remark}
This brings us to our thoughts towards addressing Question \ref{handlebody geometry}.
  
\begin{conjecture} There exists $C>0$ such that if $H^i_0, H^i_1$ are as in Remark \ref{goal}, then by switching if necessary $H^i_0$ and $H^i_1$,  if $D$ is a compressing disc for $H^i_0$, then $\length(\partial D)>C g$.  \end{conjecture}

\noindent\textbf{Zipped up Handlebodies.}  Let $K$ be a compact connected 1-complex each of whose edges has length 1.  Let $V_0$ denote a small 3-dimensional neighborhood of $K$ with $\{D_j\}$ a  family of \emph{standard} compressing discs for $V_0$, one for each edge of $K$.  Let $\gamma\subset \partial V_0$ denote an embedded arc which intersects each $D_j$ at least twice.  Let $z_0$ be the midpoint of $\gamma$.  Parametrize  by arc length the closure $\alpha_0, \alpha_1$ of the components  of $\gamma\setminus z_0$ by $[0,L]$, where $\alpha_0(0)=\alpha_1(0)=z_0$.  Identify $N(\alpha_i)\subset \partial V_0$ with $[0,L]\times [0,\epsilon]$.  Let $B=[0,L-1]\times [0,\epsilon]\times [0,2]$.  Now glue $B$ to $V_0$ to obtain $V_1$ so that, after rounding corners  $\gamma\times [0,\epsilon]$ is identified with $[0,L-1]\times [0,\epsilon]\times \{0,2\}\cup 0\times [0,\epsilon]\times [0,2]$  where $z\times [0,\epsilon]$ is identified with $0\times [0,\epsilon]\times 1$.  See Figure A.1.   If the gluing is essentially arc length preserving, then metrically, $V_1$ is close to the handlebody obtained by gluing $\alpha_0\times[0,\epsilon]$ to $\alpha_1\times[0,\epsilon]$. We say that $V_1$ is obtained from $V_0$ by  \emph{zipping along $\gamma$} and that $V_1$ is a \emph{zipped up handlebody}.

\begin{remark} It is not difficult to construct such a $\gamma$ that hits each $D_j$ exactly two times.\end{remark}

\begin{figure}[ht]
\includegraphics[scale=0.60]{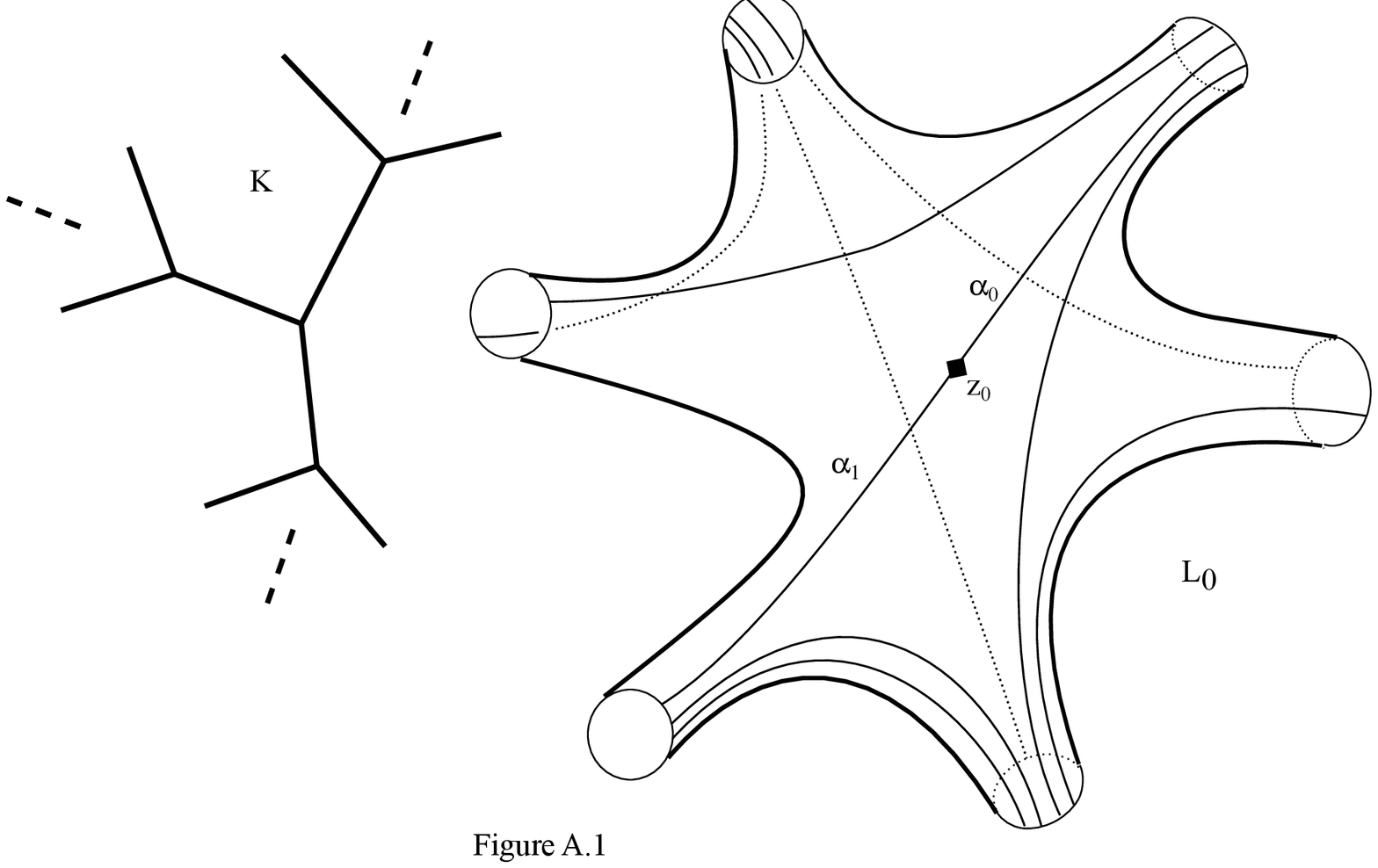}
\end{figure}
The \emph{obvious} compressing disc $E_i\subset V_1$  arising from $D_i$ is obtained by extending $D_i\subset V_0$ into $B$ and attaching a feeler from each component of $B\cap \partial D_i$ to $(L-1)\times [0,\epsilon]\times [0,2]$.  See Figure A.2.  It follows that  $\length(\partial E_i)$ is approximately twice the total length of its feelers.   If $\genus(V_0)=g$, and $K$ is degree 3, then $K$ has about $3g$ edges.  So, if $\gamma$ intersects each $D_i$ $8$ times, then the average length of a $\partial E_i$ is about $96 g$. To see this observe that $L$ is approximately $12 g$, so the average feeler is of length $6g$.  
\begin{figure}[ht]
\includegraphics[scale=0.60]{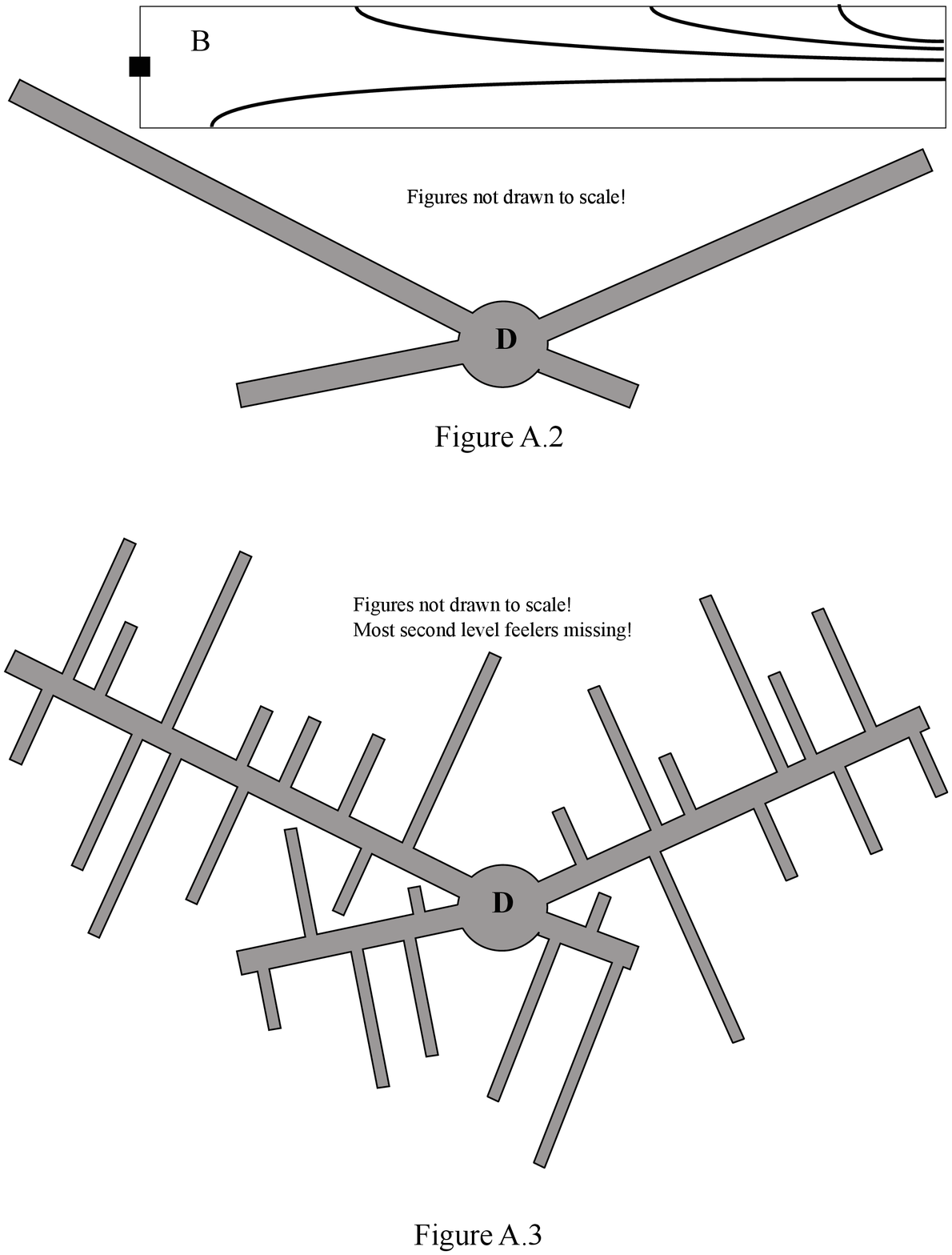}
\end{figure}

\begin{question} Is it possible to construct evenly distributed $\gamma$'s?  If so, how?\end{question}

\begin{question} For $g$ sufficiently large, does there exist a genus-g zipped up handlebody $V_1$ arising from a degree-3 graph and a length $<24 g$ zipping arc, such that if $E$ is a compressing disc, then $\length(\partial E)>10g$?\end{question}

\begin{remark}  Notice that zipping up a handlebody $V_0$  adds a small multiple of the zipping arc length to its volume.   Thus it costs relatively little, measured by volume, to go from the conventional looking handlebody to one with very complicated compressing discs.  \end{remark}

\noindent\textbf{Disclaimer.}  Here and in what follows, we are very heuristic regarding the metric which determines the volume of the handlebody and the boundary length of compressing discs.  The metric should be natural with respect to the construction, e.g. a combinatorial metric where volume is measured by the number of 3-simplices and length by intersection number with the 1-skeleton on the boundary.   We guess that if $\gamma$ goes over all the $D_i$'s many times, say approximately $r$, and $V_1$ is the zipped up handlebody, then a hyperbolic handlebody $H$ with index-$\le 1$ boundary that \emph{approximates} $V_1$ would have a fairly thick $V_0$, i.e. the radius of the $D_i$'s would be O$(\log r)$.  
\vskip 10pt

\noindent\textbf{Multi-Zipped up Handlebodies} We can generalize zipping as follows.  Let $M$ be a manifold with boundary and $\gamma\subset \partial M$ an embedded arc.  Zip up $M$ along $\gamma$ by gluing in a $B=[0,L-1]\times [0,\epsilon]\times [0,2]$  as we did with handlebody zipping.  Thus if $M$ is the handlebody $V_0$, then we can zip up $V_0$ $k$  times to obtain $V_1, V_2, \cdots, V_k$.  If $D$ is a standard compressing disc for $V_0$, then the \emph{obvious} associated compressing disc $F_2\subset V_2$ is a spiny 2-disc as in Figure A.3.

\begin{question} Fix $n\in \BN$.  For $g$ sufficiently large is it possible to construct a genus-g handlebody $V_g$ arising from a multi-zipped thickened  degree-3 graph, such that $\vol(V_g)<1000n g$ and for every compressing disc $D$, $\length(\partial D)>g^n$.\end{question}

\begin{remark}  We see this phenomena both in minimal and normal surface theory.  Suppose $M$ has the Heegaard splitting $(H_0, H_1)$ with the index-1 Heegaard surface $S$.  Let the mean convex surface $S_0$ be obtained by  pushing $S$ slightly into the $H_0$.  Suppose that $S_0$  becomes extinct after applying mean curvature flow for finite time.   Reversing the process we see $H_0$ built up from balls, 1-handles and expansion.  The expansion may \emph{correspond} to zipping.  We also see this in normal surface theory.  If $S$ is an almost normal Heegaard surface, then normalization to the  $H_0$ side collapses $H_0$ to $H'_0$ where $\partial H'_0$ is a possibly empty normal surface.  Reversing the process one may see zipping in addition to creation of 0 and 1-handles.\end{remark}  

\begin{question}  Let $\{N_i\}$ be a tower $\mT$ of covers as  in Theorem \ref{tau}.  Let $mg_i(\mT)$ be the number of distinct minimal genus Heegaard splittings of $N_i$.  What is $\liminf mg_i/d_i$ where $N_i\to N_0$ is a degree-$d_i$ cover?  Fix an irreducible Heegaard splitting $\mH$ of $N_0$.  Let $\mH_i$ denote its preimage in $N_i$ and $r_i$ denote the number of distinct irreducible Heegaard splittings of $N_i$ obtained by destabilizing $\mH_i$.  What is $\liminf r_i/d_i$?\end{question}

\noindent\textbf{Summary.} In this appendix we presented constructions of geometrically different families of handlebodies.  Through the multi-zipped handlebody construction, we offer a conjectural approach towards Question \ref{handlebody geometry}, including ideas for constructing, for $g$ sufficiently large and $n$ fixed, hyperbolic or simplicial handlebodies $V_g$ of genus $g$ with $\volume (V_g)<C_0 g$ such that all compressing discs have $\length > C_1 g^n$.

\end{document}